\documentclass[reqno]{amsart}

\usepackage{amsmath}
\usepackage{amsthm}
\usepackage{amsfonts}
\usepackage{amssymb}
\usepackage{enumitem}
\usepackage{hyperref}
\usepackage{bbm}
\usepackage[nocompress]{cite}

\newcommand{\oindicator}[1]{\ensuremath{\mathbf{1}_{{#1}}}}

\DeclareMathOperator{\tr}{tr}

\DeclareMathOperator{\ii}{i}

\newcommand{\E}{\mathbb{E}}
\newcommand{\C}{\mathbb{C}}
\renewcommand{\P}{\mathbb{P}}
\newcommand{\N}{\mathbb{N}}
\renewcommand\Re{\operatorname{Re}}
\renewcommand\Im{\operatorname{Im}}
\newcommand{\eps}{\varepsilon}

\def\R{\mathbb{R}}

\newcommand{\dd}{\mathrm{d}}

\theoremstyle{plain}
\newtheorem{theorem}{Theorem}[section]

\newtheorem{lemma}[theorem]{Lemma}
\newtheorem{corollary}[theorem]{Corollary}

\newtheorem{proposition}[theorem]{Proposition}

\theoremstyle{definition}
\newtheorem{definition}[theorem]{Definition}

\theoremstyle{remark}
\newtheorem{remark}[theorem]{Remark}

\makeatletter \let\@wraptoccontribs\wraptoccontribs \makeatother

\begin{document}
	
	\title[Extreme eigenvalues of Laplacian random matrices]{Universality for the extreme eigenvalues of Laplacian random matrices}

	\author[A. Campbell]{Andrew Campbell}
	\address{Department of Mathematical Sciences\\ Lafayette College\\ 730 High Street\\ Easton, PA 18042}
	\email{campband@lafayette.edu}
	
	\author[K. Luh]{Kyle Luh}
	\address{Department of Mathematics\\ University of Colorado\\ Campus Box 395\\ Boulder, CO 80309-0395\\USA}
	\email{kyle.luh@colorado.edu}
	\thanks{K. Luh was supported in part by the Simons Foundation Grant MP-TSM-00001988.}
	
	\author[S. O'Rourke]{Sean O'Rourke}
	\address{Department of Mathematics\\ University of Colorado\\ Campus Box 395\\ Boulder, CO 80309-0395\\USA}
	\email{sean.d.orourke@colorado.edu}
	\thanks{S. O'Rourke has been supported in part by NSF CAREER grant DMS-2143142. }

	\begin{abstract}
	We study the eigenvalues of the random Laplacian matrix $D - A$, where $A$ is a Wigner matrix with sub-Gaussian entries and the diagonal matrix $D$ contains the row sums of $A$. 
	Our main results show that the extreme eigenvalues of this model exhibit Poisson statistics; in particular, after the appropriate centering and scaling, the largest eigenvalue converges to the Gumbel distribution as the dimension of the matrix tends to infinity. 
	This confirms, for general sub-Gaussian entries, a phenomenon the authors previously established only in the Gaussian case [Electron.\ J. Probab. 30 (2025), Paper No. 104], resolving a conjecture raised there.
	As a corollary, for an Erd\H{o}s--R\'enyi random graph, we show the asymptotic fluctuations of the algebraic connectivity (Fiedler value) can be described in terms of the Gumbel distribution. 	
	\end{abstract}

	\maketitle

	\section{Introduction}
	
	For symmetric matrices with independent (or weakly dependent) entries, random matrix theory provides strong universality results for the eigenvalues. 
	This includes the global semicircle law down to the optimal local scale \cite{MR1727234,MR2669449,MR2784665,MR3916110,MR2730790,MR2981427,MR3068390,MR3372074,MR3478311,MR3541852,MR3699468,MR2567175,MR2906465,MR4550215}. 
	For instance, the largest eigenvalues asymptotically converge to the appropriate Airy point process, and the largest eigenvalue is known to converge to the Tracy-Widom law as long as the entries have light enough tails \cite{MR3800840,MR4019881,MR4091114,MR2485010,MR2545551,MR2398763,MR3582818,MR2988403,MR4260217,MR2308592,MR4474568,MR3161313,MR1257246,MR1385083}.  
	The situation is different for random matrices with heavy tails.  
	In this case, the extreme eigenvalues instead converge to a Poisson point process dictated by the largest entries of the matrix \cite{MR2081462,MR2125574,MR2234922,MR2373441,MR2375601,MR2548495}.  
	Much work has been done in the areas discussed above, and the references provided are not a complete listing. 
	
	For the largest eigenvalues, we will show that Laplacian matrices exhibit a behavior closer to matrices with heavy-tails, even when the entries are sub-Gaussian. 
	
	\begin{definition} \label{def:lap}
		For an $n\times n$ real symmetric matrix $A=(a_{ij})$, we define the \emph{Laplacian matrix} $L_A$ of $A$ as \begin{equation}
			{L}_{A}=D_{A}-A,
		\end{equation}  where $D_{A}=(D_{ij})$ is the diagonal matrix containing the row sums of $A$:\begin{equation}
		D_{ii}=\sum_{j=1}^{n}a_{ij}.
		\end{equation}
	\end{definition} 
	
	Random Laplacian matrices arise in a number of applications including those involving complex graphs \cite{MR2248695},  community detection in the stochastic block model \cite{MR3349181}, the $\mathbb{Z}_2$ synchronization problem \cite{MR3777782}, and other optimization problems in semidefinite programming \cite{MR3777782}.  Random Laplacian matrices have also been used in the theoretical physics literature to study random impedance networks \cite{MR1747178, PhysRevE.67.047101}.  
	
	We will focus on the case when $A$ is a Wigner matrix with arbitrary diagonal entries.
	\begin{definition} \label{def:Wigner matrix}
		Let $\xi$ be a real random variable with mean zero and unit variance.  
		An $n \times n$ symmetric random matrix $A = (a_{ij})$ is a \emph{Wigner matrix} if:
		\begin{enumerate}
		\item (off-diagonal entries) $a_{ij}$, $1 \leq i < j \leq n$ is a collection of iid copies of $\xi / \sqrt{n}$.
		\item (flexibility of the diagonal) $a_{ii}$, $1 \leq i \leq n$ is a collection of arbitrary random variables which may be dependent on one another, may depend on  the off-diagonal entries, and may depend on $n$.
		\end{enumerate}
		We call $\xi$ the \emph{atom variable} of $A$. 
	\end{definition}
	\begin{remark}
	Classically, Wigner matrices require the diagonal entries to be independent of one another and the off-diagonal entries.  Since the diagonal entries of $A$ do not contribute to the Laplacian $L_A$, we need not make any assumptions on the diagonal entries. In fact, we will assume the diagonal entries are zero for the entire proof. 
	\end{remark}
	
	Even when the entries of $A$ are light-tailed, the diagonal entries of $D_A$ dominate over the entries of $A$, as each nonzero entry of $D_A$ is a sum of $n-1$ independent random variables.  
		From the central limit theorem one expects the diagonal entries of $D_A$ to be Gaussian. 
	In fact, under some appropriate moment assumptions on the atom variable, when $A$ is a Wigner matrix the limiting empirical spectral distribution for $L_A$ is given by the free convolution of the semicircle law and the Gaussian distribution \cite{MR2759729,MR2206341,MR4440252,MR1747178, PhysRevE.67.047101}. 
	These results have been extended to generalized Wigner matrices \cite{MR4440252}, inhomogeneous and dilute Erd\H{os}--R\'{e}nyi random graphs \cite{MR4193182,MR2967963}, and block Laplacian matrices \cite{MR3321497}.

	We note that since the free convolution of the semicircle law and the Gaussian distribution is supported on the entire real line, there is not as natural a choice of ``edge'' as there is in the case of compactly supported limiting empirical spectral distributions. 
	 For a large class of Wigner matrices $A$, the location of the largest eigenvalue of $L_A$ has been studied by Ding and Jiang \cite{MR2759729}.  
	 Their results show that the largest eigenvalue of $L_A$ is asymptotically close to $\sqrt{2 \log n}$ as the dimension $n$ tends to infinity.  
	 The generalized Wigner case was considered in \cite{MR4440252}, while the smallest eigenvalues were studied in \cite{MR2948689}.  Spectral norm bounds are established in \cite{MR2206341,MR4440252}. 
	 
	The goal of this paper is to study the largest eigenvalues of $L_A$ when $A$ is a Wigner matrix with sub-Gaussian atom variable.   
	While we focus on the edge of the spectrum, the eigenvalues in the bulk (and corresponding eigenvectors) were studied by Huang and Landon \cite{MR4058984} for the case when $A$ is a Wigner matrix or the adjacency matrix of a sparse Erd\H{os}--R\'{e}nyi random graph.  
		
		By classical extreme value theory, it is natural to expect the largest eigenvalues to be determined by the largest diagonal entries of $L_A$. 
		The authors established this phenomenon in \cite{campbell2022extreme} in the case when $A$ is drawn from the Gaussian Orthogonal Ensemble (GOE).  
		Indeed, it was shown that the largest eigenvalues converge to a Poisson point process, and the largest eigenvalue's fluctuations are asymptotically Gumbel.  
		The present paper confirms that this phenomenon is universal, resolving a conjecture from \cite{campbell2022extreme}.
		While the arguments in \cite{campbell2022extreme} relied on the rotational invariance of the GOE, we give a universal proof below that applies to a wide class of Wigner matrices $A$ with sub-Gaussian entries.
		Recall that a random variable $\xi$ is \emph{sub-Gaussian} if there exists $K>0$ such that 
		\begin{equation}
			\P(|\xi|\geq t)\leq 2\exp(-t^2/K^2),
		\end{equation} 
		for all $t\geq 0$, or equivalently there exists some $K'>0$ such that 
		\begin{equation}\label{eq:MGF def of subgaussian}
			\E\exp(\xi^2/K'^2)\leq 2.
		\end{equation} The sub-Gaussian norm of $\xi$ is defined by 
		\begin{equation}\label{eq:subgaussian norm}
			\|\xi\|_{\psi_2}=\inf\{K> 0:\E\exp(\xi^2/K^2)\leq 2\}.
		\end{equation}

	\subsection{Main results} 
	Define the sequences 
	\begin{equation}\label{eq:eigen scale and shift}
		a_n=\sqrt{2\log n},\quad b_n=\sqrt{2\log n}-\frac{\log\log n+\log(4\pi)-2}{2\sqrt{2\log n}},
	\end{equation} and \begin{equation}\label{eq:gauss shift}
	b_n'=\sqrt{2\log n}-\frac{\log\log n+\log(4\pi)}{2\sqrt{2\log n}}.
	\end{equation}  If $Z_1,Z_2,\dots$ are independent standard Gaussian random variables, then $a_n$ and $b_n'$ are classically known to be such that for any $x\in\R$ \begin{equation}
	\lim\limits_{n\rightarrow\infty}\P\left(a_n \left[\max_{1\leq j \leq n}Z_{j}-b_n' \right] \leq x \right)=\exp\left(-\exp(-x)\right)=:F(x),
	\end{equation} 
	where $F$ is the cumulative distribution function of the standard \emph{Gumbel distribution} \cite{MR691492,MR900810}. 
	
	For the eigenvalues of the Laplacian matrix, it turns out that $b_n$, rather than $b_n'$, is the correct centering term.  
	Our first main result describes the fluctuations of the largest eigenvalue.  
	
	\begin{theorem}[Largest eigenvalue] \label{thm:largest eigenvalue}
		Let $\{A_{n}\}_{n\geq1}$ be a sequence of Wigner matrices indexed by their dimension with common sub-Gaussian atom variable $\xi$. If $\lambda_{1}(L_{A_n})$ is the largest eigenvalue of $L_{A_n}$, then $a_n(\lambda_{1}(L_{A_n})-b_n)$ converges in distribution to the Gumbel distribution as $n\rightarrow\infty$. 
	\end{theorem}
	
	More generally, we also consider the point process formed by the largest eigenvalues of the Laplacian.  
	
	\begin{theorem}[Poisson point process limit] \label{thm:eigenvalue point process}
		Let $\{A_{n}\}_{n\geq1}$ be a sequence of Wigner matrices indexed by their dimension with common sub-Gaussian atom variable $\xi$. 
		Define the point process 
		\begin{equation}
			\mathcal{L}_{n}=\sum_{j=1}^{n}\delta_{a_n ( \lambda_{j}(L_{A_n}) - b_n) },
		\end{equation} 
		where $\delta_x$ is a Dirac point mass at $x \in \mathbb{R}$. 
		Then $\mathcal{L}_n$ converges in distribution in the vague topology to a Poisson point process with intensity measure $\dd \mu=e^{-x}\dd x$ as $n \to \infty$.
	\end{theorem}  

	As with all of our main results, Theorems \ref{thm:largest eigenvalue} and \ref{thm:eigenvalue point process} can be adapted to the smallest eigenvalues as well.  Indeed, by considering the Wigner matrix $-A$ with atom variable $-\xi$, one can establish analogous results for the smallest eigenvalues.

	Both Theorems \ref{thm:largest eigenvalue} and \ref{thm:eigenvalue point process} follow from our next main result. 
		
	\begin{theorem}[$k$ largest eigenvalues] \label{thm:largest k eigenvalues}
		Let $\{A_{n}\}_{n\geq1}$ be a sequence of Wigner matrices indexed by their dimension with common sub-Gaussian atom variable $\xi$, and let $k$ be a fixed positive integer.  If $\{Z_{j}^{(n)}\}_{1\leq j\leq n}$ is a triangular array of independent standard Gaussian random variables with order statistics  $Z_{(1)}^{(n)}\geq Z_{(2)}^{(n)}\geq\cdots\geq Z_{(n)}^{(n)}$ in level $n$, then for any $x_1, \ldots, x_k \in \R$
        \begin{equation}\label{eq:Gumbel limit stats joint}
			\begin{aligned}
				F_{k}(x_1,\dots,x_k)&:=\lim\limits_{n\rightarrow\infty}\P\left(a_n\left(Z_{(1)}^{(n)}-b_n' \right)\leq x_1,\dots,a_n\left(Z_{(k)}^{(n)}-b_n' \right)\leq x_k \right)\\
				&\ =\lim\limits_{n\rightarrow\infty}\P\left(a_n\left(\lambda_{1}(L_{A_n})-b_n \right)\leq x_1,\dots,a_n\left(\lambda_{k}(L_{A_n})-b_n \right)\leq x_k \right),
			\end{aligned}
		\end{equation}
		where $\lambda_1(L_{A_n}) \geq \lambda_2(L_{A_n}) \geq \cdots \geq \lambda_n(L_{A_n})$ are the ordered eigenvalues of $L_{A_n}$. 
	\end{theorem}
	
	Theorem \ref{thm:largest k eigenvalues} implies Theorem \ref{thm:eigenvalue point process}, see \cite[Proof of Corollary 1.6]{campbell2022extreme} for the proof. 
	Theorem \ref{thm:largest eigenvalue} follows from Theorem \ref{thm:largest k eigenvalues} by taking $k = 1$ and applying classical results (see, for example, \cite[Theorem 1.5.3]{MR691492}).   
	
	We conjecture that our main results should hold even when the atom variable is not sub-Gaussian, provided it has finite fourth moment. 
    We have not tried to optimize the assumptions on the atom variable, and there are several points in the proof that do use the sub-Gaussian assumption. 
		
	All of our above results can be extended to Wigner matrices whose atom variable has  non-zero mean. The key observation is the following.  Without loss of generality, we can assume the Wigner matrix $A_n$ has all zero diagonal entries, and if $\overline{A}=\E A_n$, then
    \begin{equation*}
		\overline{A}=\E a_{12}\mathbf{e}\mathbf{e}^{\mathrm{T}}-\E a_{12}I_n,
	\end{equation*} where the all-ones vector $\mathbf{e}$ is  an eigenvector of any Laplacian matrix. Thus, the map between Laplacian matrices defined by $D-A\mapsto D-(n-1)\E a_{12}I_n-A+ \overline{A}$ results in a simple deterministic shift of the non-trivial eigenvalues. We illustrate this extension in detail in the next theorem, where we compute the fluctuations of the spectral gap (also known as the algebraic connectivity or Fiedler value) of an Erd\H{o}s--R\'enyi random graph. 
	
	\begin{theorem}\label{thm:Spectral gap for Laplacian}
		Let $p\in (0,1)$ be fixed and $L_n$ be the Laplacian of an Erd\H{o}s--R\'enyi random graph\footnote{In our definition of $G(n,p)$, self-loops are not permitted, but this difference is irrelevant for the Laplacian.} $G(n,p)$. If $\lambda_{n-1}(L_n)$   is the second smallest eigenvalue of $L_n$, i.e.\ the spectral gap, then for any $x \in \mathbb{R}$
        \begin{equation}
			\lim\limits_{n\rightarrow\infty}\P\left(a_n\left[\sqrt{\frac{np}{1-p}}-b_n-\frac{\lambda_{n-1}(L_n)}{\sqrt{np(1-p)}}  \right]\leq x \right)=\exp\left(-e^{-x}\right).
		\end{equation} 
	\end{theorem}
	
	\begin{proof}
		Define the Laplacian matrix \begin{equation}\label{eq:ER graph matrix transform}
			\hat{L}_n=\frac{-1}{\sqrt{np(1-p)}}\left(L_n-npI_n+p\mathbf{e}\mathbf{e}^{\mathrm{T}}\right).
		\end{equation} $\hat{L}_n$ is the Laplacian of the Wigner matrix $\frac{1}{\sqrt{np(1-p)}}(p\mathbf{e}\mathbf{e}^{\mathrm{T}}-pI_n-A)$, where $A$ is the adjacency matrix of $G(n,p)$. This Wigner matrix satisfies the assumptions of Theorem \ref{thm:largest eigenvalue}. Because $\mathbf{e}$ is the eigenvector of the trivial eigenvalue of $L_n$ and $\hat{L}_n$ we have that \begin{equation}\label{eq:ER graph gap transform}
		\lambda_{1}(\hat{L}_n)=\frac{1}{\sqrt{np(1-p)}}\left(np-\lambda_{n-1}(L_n)\right),
		\end{equation} 
        provided the right-hand side is non-negative (which happens with probability tending to one by Theorem \ref{thm:largest eigenvalue}). 
        The conclusion then follows from Theorem \ref{thm:largest eigenvalue}.
	\end{proof}

	\subsection{Overview}
	The rest of the paper is devoted to the proof of Theorem \ref{thm:largest k eigenvalues}. 
    When there is no confusion, we will use $A$ to denote the $n \times n$ matrix $A_n$. 
    Without loss of generality, we will assume the diagonal entries of $A$ are all zero. 
	In \cite{campbell2022extreme}, the authors studied the case when $A$ is drawn from the GOE; in particular, they exploited the rotational invariance of the GOE in order to decouple the dependence between $A$ and $D_A$. 
	Unfortunately, this initial step is not available when $A$ is an arbitrary Wigner matrix.  
	To get around this problem, we instead use the following two step procedure to prove Theorem \ref{thm:largest k eigenvalues}:
	\begin{enumerate}
	\item[Step 1:] We show the largest entries of $D_A$ have the same asymptotic behavior as the largest order statistics from a sample of iid standard Gaussian random variables.  
	Although this may sound very similar to a classically studied problem, the fact that the entries of $D_A$ change with $n$ and are dependent on one another makes the problem more difficult. 
	The proof requires a careful approximation of the entries of $D_A$ by independent sums of random variables.  
	\item[Step 2: ] We show the largest eigenvalues of $L_A$ are dictated by the largest entries of $D_A$.  
	In particular, we study the Stieltjes transform $m_n(z) = \frac{1}{n} \tr (L_A - z)^{-1}$ where $z \in \mathbb{C}$ with $\Im(z) > 0$.  
	We take $\Re(z) \approx \sqrt{2 \log n}$ which is where we expect the largest eigenvalues to be. 
	In doing so, we derive a self-consistent equation and local law for $m_n$ near the edge of the spectrum which explicitly expresses $m_n$ in terms of the diagonal entries of $D_A$.
	This local law allows us to estimate the locations of the largest eigenvalues in terms of the largest diagonal entries of $D_A$, up to some negligible error terms.  
	\end{enumerate}
	Combining the two steps will complete the proof of Theorem \ref{thm:largest k eigenvalues}. 
	
	We introduce the notation used in the paper in Section \ref{sec:notation} and the main tools in Section \ref{sec:tools}. 
	Section \ref{sec:largest-entries} establishes Step 1.  Step 2 and  the proof of Theorem \ref{thm:largest k eigenvalues} are contained in Section \ref{sec:proof}.

	\section{Notation} 	\label{sec:notation}
	For a complex number $z$, we will use $\Re(z)$ and $\Im(z)$ to denote the real and
	imaginary part, respectively.  We will use $i$ to denote the imaginary unit and also as
	an index of summation; the context should make this clear.  We also let $\mathbb{C}_+$ be the 
	upper-half plane.
	
	For a natural number $N$, $[N]$ is defined to be the discrete interval 
	$\{1, \cdots, N\}$.  For two real numbers $a, b$ with $a < b$, 
	we let $[a,b]$ denote the set $\{r: a \leq r \leq b\}$.  
	
	For a matrix $M$, $M_{ij}$ is the $(i,j)$-entry.  The matrix transpose of $M$ is denoted
	$M^\mathrm{T}$. $\tr(M)$ is the trace of $M$ and for symmetric $M$, we let
	\[
	\lambda_n(M) \leq \cdots \leq \lambda_1(M)
	\]   
	be its eigenvalues.  For arbitrary $M$, $\|M\|$ will denote the operator norm and $I$ the 
	identity matrix.  For a vector $v$, $\|v\|_2$ will mean the standard Euclidean norm. 
	
	For an event $E$, $\P(E)$ is the probability of the event.  We use $\mathrm{1}_E$ 
	to be the indicator function of the event $E$.  For a random variable $\xi$, $\E \xi$
	is the expectation.  $\Phi$ will denote the cumulative distribution function for the standard 
	normal random variable.   
	
	As our parameter $n$ tends to infinity, we use the following asymptotic notation.  
	We use $X = O(Y)$ or $X \simeq Y$ to mean that $|X| \leq C Y$ for some constant $C >0$
	independent of $n$ for all $n \geq C$.  $X = o(Y)$ will mean that $|X| \leq c_n Y$ for some sequence $(c_n)$ that converges to zero as $n \rightarrow \infty$.  
	
	Now, we discuss the notation specific to our model. 
	We will drop the $A$ dependence from the diagonal and Laplacian matrix when the context is clear, simply writing $L = D-A$.
	Let $L^{(k)}$ be the matrix obtained by removing the $k$-th row and column from $L$.
	We define the resolvent of $L$ to be
	$$
	G(z) := (L - z)^{-1}.
	$$
	We let 
	\[
	G^{(k)}(z) := (L^{(k)}-z)^{-1}.
	\]
	Note that $L^{(k)}$ still depends on elements of the $k$-th row and column of $L$.  To address this, we define $\tilde{L}^{(k)}$ to be the matrix with entries
	\[
	\tilde{L}^{(k)}_{ij} := L^{(k)}_{ij} - \delta_{ij} a_{ik} \text{ \qquad for } i,j \neq k.
	\]
	Analogously, we define
	\[
	\tilde{G}^{(k)}(z):= (\tilde{L}^{(k)} - z)^{-1}.
	\]
	We define the corresponding Stieltjes transforms to be
	$$
	m_n(z) := \frac{1}{n} \tr G(z)
	$$
	$$
	m_n^{(k)}(z) := \frac{1}{n-1} \tr G^{(k)}(z)
	$$
	$$
	\tilde{m}_n^{(k)}(z) := \frac{1}{n-1} \tr \tilde{G}^{(k)}(z) 
	$$
	Finally, we define the Stieltjes transform of the free convolution of the semicircle 
	law with the standard normal distribution to be
	$$
	m(z) = \int_{-\infty}^\infty \frac{\phi(x)}{x - z - m(z)} \, dx.
	$$

	For a fixed small $\delta>0$ we consider the ($n$ dependent) spectral domains \begin{equation}
		\begin{aligned}
			S_{\delta}&:=\left\{z\in\C_+: \sqrt{(2-\delta)\log n}\leq \Re z\leq\sqrt{3\log n}, n^{-1/4}\leq \Im z\leq 1 \right\}\\
			\tilde{S}_{\delta}&:=\left\{z\in\C_+: \sqrt{(2-\delta)\log n}\leq \Re z\leq\sqrt{3\log n}, \Im z=n^{-1/4} \right\}\\
			\hat{S}_{\delta}&:=\left\{z\in\C_+: \sqrt{(2-\delta)\log n}\leq \Re z\leq\sqrt{3\log n}, \Im z=\sqrt{2}n^{-1/4} \right\}.\\ 
		\end{aligned}
	\end{equation}

	\section{Tools} \label{sec:tools}
	This section lists some basic tools and results that are needed for the proof of the main results.
	We will use the following basic structural result about the inverse of block matrices.
	\begin{lemma}[Schur's Complement Formula] \label{lem:schur}
		Assuming all the inverse matrices exist, the inverse of the block matrix is given by
		\begin{equation}
			\left( \begin{array}{cc} A & B \\ C & D \end{array} \right)^{-1} = 
			\left( \begin{array}{cc} \mathcal{A}^{-1} & -\mathcal{A}^{-1} B D^{-1} \\ -D^{-1} C \mathcal{A}^{-1} & D^{-1} C \mathcal{A}^{-1} B D^{-1} + D^{-1} \end{array}\right)
		\end{equation} 
		where 
		$$
		\mathcal{A} := A - B D^{-1} C \text{\quad and \quad} \mathcal{D} := D - C A^{-1} B.
		$$
	\end{lemma}

	For any invertible matrices $A$ and $B$, the \emph{resolvent identity} states that \begin{equation}
			A^{-1}-B^{-1}=A^{-1}(B-A)B^{-1}.
		\end{equation} Additionally, the \emph{Ward identity} states that for any resolvent matrix $G(z)=(M-z)^{-1}$, where $M$ is an $n\times n$ Hermitian matrix and $z\in\C_{+}$, that \begin{equation}
		\sum_{j=1}^{n}|G_{ij}(z)|^2=\eta^{-1}\Im G_{ii}(z).
		\end{equation} 

	We recall the following estimate on the size of the imaginary component of $m(z)$ 
	in the spectral domains defined above.

	\begin{lemma} \label{lem:immz} \cite[Lemma 3.20]{campbell2022extreme}
		Fix $0 < \delta < 2/9$.  Then, 
		\[
		\sup_{E + \ii \eta \in \tilde{S}_\delta \cup \hat{S}_\delta} \Im m(z) = o(n^{-1/4}).
		\]
	\end{lemma}
	
	The next lemma controls the distances between Stieltjes transforms when 
	removing rows and columns.   
		\begin{lemma}[Lemma 7.5, \cite{MR3210147}] \label{lem:rank1}
		\[
		|m_n(z) - m_n^{(k)}(z)| \leq \frac{1}{n \eta}.
		\]
	\end{lemma}

	We now reference a couple large deviation results which will of use.  
	
	\begin{lemma} \label{lem:largedev}
		Let $X_1, \dots, X_n$ be independent mean-zero sub-Gaussian random variables with $\|X_i\|_{\psi_2} \leq K$, and let $a = (a_1, \dots, a_n) \in \R^n$.  Then, for every $t \geq 0$, we have
		\[
		\P \left( \left| \sum_{i=1}^n a_i X_i \right| \geq t \right)  \leq 2 \exp \left(- \frac{c t^2}{K^2 \|a\|_2^2} \right)
		\]
	\end{lemma}
	
	\begin{lemma} [Hanson-Wright Inequality] \label{lem:HW}
		Let $X_1, \dots, X_n$ be independent mean-zero sub-Gaussian random variables with $\|X_i\|_{\psi_2} \leq K$.  Let $A$ be an $n \times n$ matrix.  Then,
		\[
		\P\left( |X^\mathrm{T} A X - \E X^\mathrm{T} A X| \geq t\right) \leq 2 \exp \left[ -c \min \left(\frac{t^2}{K^4 \|A\|_{HS}^2}, \frac{t}{K^2 \|A\|} \right) \right]
		\]
	\end{lemma}
	
	\begin{lemma}[Theorem 6.3.2 in \cite{Vershynin2018}]\label{lem:concentration of vec-matrix}
		Let $B$ be an $m\times n$ matrix and $X$ an $n$-dimensional random vector with independent, mean zero, unit variance, sub-Gaussian coordinates. Then,\begin{equation}
			\left\| \|BX\|_{2}-\|B\|_{F} \right\|_{\psi_2}\leq CK^2\|B\|,
		\end{equation}  where $\|\cdot\|_{F}$ is the Frobenius norm of a matrix and $K$ is the max of the sub-Gaussian norms of the entries of $X$. 
	\end{lemma}

	The next lemma is a consequence of  \cite[Theorem 1, Chapter VIII]{petrov1972independent}
	and provides a universality statement for moderate deviations.  
	\begin{lemma}\label{lem:petrov's lemma}
		Let $S_n = X_1 + \dots + X_n$ where $X_i$ are iid,
		\[
		\E X_i = 0,\quad  \E X_i^2 = 1
		\]
		and $\E e^{\theta |X_i|} < \infty$ for some $\theta > 0$. Then, for every $C > 0$,  
		\[
		\sup_{0 \leq t \leq C \sqrt{\log n}} \left| \frac{\P(S_n/ \sqrt{n} > t)}{1 - \Phi(t)} - 1 \right| \rightarrow 0.
		\]
	\end{lemma}
	The final lemma in this section is a simple calculation.
	\begin{lemma}\label{lem:Gaussian tail max idk}
		For any compact set $K \subset \mathbb{R}$, 
		\[
		\sup_{x \in K} |n ( 1- \Phi(u_n(x)))- e^{-x}| \rightarrow 0,
		\] where $u_n(x)=b_n'+\frac{x}{a_n}$. 
		Moreover, if $d_n:K\rightarrow\R$ is such that $\sup_{x \in K} a_n |d_n(x)| \rightarrow 0$ as $n\rightarrow\infty$, 
		then
		\[
		\sup_{x \in K} |n (1 - \Phi(u_n(x) + d_n(x))) - e^{-x} | \rightarrow 0.
		\]
	\end{lemma}
	\begin{proof}
		Let $c_n = \frac{\log \log n + \log (4 \pi)}{2}$ and observe that $u_n(x) = a_n + \frac{x - c_n}{a_n}$.  For $x \in K$,
		\[
		u_n(x)^2 = 2 \log n + 2 (x - c_n) + o(1).
		\]
		Note that
		\[
		1 - \Phi(u_n(x)) = \frac{1 + o(1)}{u_n(x) \sqrt{2 \pi}} \exp(-u_n(x)^2/2).
		\]
		Since $u_n(x) \rightarrow \sqrt{2 \log n}$ and $e^{c_n} = \sqrt{4 \pi \log n}$,
		\[
		n ( 1- \Phi(u_n(x))) = (1 + o(1)) \frac{e^{c_n - x}}{\sqrt{4 \pi \log n}} \rightarrow e^{-x}.
		\]
		The same argument applies to $u_n(x)+d_n(x)$. 
	\end{proof}


\subsection{Counting Lemma}
  A key element of our later arguments is a counting result on the number of diagonal elements 
  near the critical window.

		\begin{proposition}\label{prop:eigen count}
		Let $L$ be the Laplacian of an $n \times n$ Wigner matrix with sub-Gaussian atom variable.  Then for any $\varepsilon \in (0,1/2)$, there exists $\delta > 0$ so that with overwhelming probability,
		\[
		\left| \{1 \leq i \leq n: L_{ii} \geq \sqrt{(2 - \delta) \log n}\} \right| = O_{\varepsilon}(n^{\varepsilon}).
		\]
		Similarly, for any $\varepsilon' \in (0,1/2)$ there exists a $\delta' > 0$ so that
		\[
		\left| \{1 \leq i \leq n : L_{ii} \geq \sqrt{ (2 - \delta') \log n}\}\right| = \Omega_{\varepsilon'}(n^{\varepsilon'})
		\]
		with overwhelming probability. Moreover, for any $\eps''\in(0,1/2)$, \begin{equation}\label{eq:weak count}
			\left|\left\{1\leq i\leq n: L_{ii}\geq (2\log n)^{\frac{1}{2}-\eps''} \right\} \right|=o\left(\frac{n\log^{\eps''}n}{\sqrt{\log n}} \right),
		\end{equation} with overwhelming probability. 
	\end{proposition}

	\subsubsection{Proof of Proposition \ref{prop:eigen count}}
	
		We will divide $[n]$ into $m := n^{\delta/10}$ consecutive blocks all of size roughly $n^{1 - \delta/10}.$ More precisely define the partition
		\[
		[n] = B_1 \sqcup \cdots \sqcup B_m
		\]
		where $s_i := |B_i|$ and $s_i = (1+o(1)) n^{1 - \delta/10}$.  For notational convenience, we work in $B_1$ for most of the proof. For $i \in B_1$, let
		\[
		L_{ii} = \sum_{j \notin B_1} a_{ij} +  \sum_{j \in B_1} a_{ij} := X_i + R_i
		\] 
		Most of the random variables are in the first sum by our choice of $m$.  
		The point of this block division is that $X_i$ are independent for $i \in B_1$.  
		
		The next lemma shows that the $R_i$'s cannot affect the count of large $L_{ii}$ by very much.
		\begin{lemma}\label{lem:neglible part in counting}
			Let $\tau_n := n^{-\delta/30}$.
			Then, with overwhelming probability,
			\[
			\max_{i \in B_1} |R_i| \leq \tau_n
			\]
		\end{lemma}
		\begin{proof}
			By \cite[Proposition 2.6.1]{Vershynin2018}, 
			\[
			\|R_i\|_{\psi_2} \leq C \sqrt{\frac{s_1}{n}} = O(n^{-\delta/20})
			\]
			so
			\[
			\P(|R_i| > \tau_n) \leq 2 \exp\left(-c n \tau_n^2/ s_1 \right) \leq n^{-K}
			\]
			for any $K$. Taking a union bound over $i\in B_1$ completes the proof.
		\end{proof}
		
		Now observe that on the event described in the above lemma,
		\[
		\mathrm{1}_{X_i \geq \sqrt{(2 - \delta) \log n} + \tau_n} \leq \mathrm{1}_{L_{ii} \geq \sqrt{(2 - \delta) \log n}} \leq \mathrm{1}_{X_i \geq \sqrt{(2-\delta) \log n} - \tau_n}
		\]
		so if we define
		\[
		N^{-} := \sum_{i \in B_1}\mathrm{1}_{X_i \geq \sqrt{(2 - \delta) \log n} + \tau_n}  
		\]
		and 
		\[
		N^{+} :=  \sum_{i \in B_1}\mathrm{1}_{X_i \geq \sqrt{(2 - \delta) \log n} - \tau_n}  
		\]
		then on the event in Lemma \ref{lem:neglible part in counting}
		\[
		N^- \leq N(\delta) \leq N^{+}
		\]
		where \begin{equation*}
			N(\delta)=\sum_{i\in B_1}  \mathrm{1}_{L_{ii} \geq \sqrt{(2 - \delta) \log n}}.
		\end{equation*} 
		
		Observe that $X_i$ has variance $\sigma_i^2 := (n-s_1)/n = 1 + O(n^{-\delta/30})$. 
		Now define
		\[
		p^{-} := \P(X_i \geq \sqrt{(2 - \delta) \log n} + \tau_n) \text{ and } p^+ := \P(X_i \geq \sqrt{(2 - \delta) \log n } - \tau_n).
		\] 
		Applying  Lemma \ref{lem:petrov's lemma} we deduce that
		\[
		p^{-} = (1 +o(1)) \left(1 - \Phi\left(\frac{\sqrt{(2 - \delta) \log n} + \tau_n}{\sigma_i} \right)\right).
		\]
		  Plugging into the Gaussian density gives
		\[
		p^{-} = (1 + o(1)) \frac{n^{-1 + \delta/2}}{\sqrt{2 \pi (2- \delta) \log n}}.
		\]

		Note that  $\E N^{\pm}= \Theta(n^{2\delta/5}/\sqrt{\log n})$.
		By Chernoff's bound,
		\[
		\P(|N^- - \E N^{-}| > t n^{2\delta/5}/\sqrt{\log n}) \leq 2 \exp(-c t^2 n^{2\delta/5}/\sqrt{\log n})
		\]
		Setting $t = n^{\delta/50}$, say, gives a superpolynomial bound on the right.  Therefore, with overwhelming probability,
		\[
		N^{-} = \Theta(n^{2\delta/5}/\sqrt{\log n}).
		\]
	 An identical argument implies $N^{+} = \Theta(n^{2\delta/5}/\sqrt{\log n})$.  Finally, take a union bound over $m$ and choose $\eps,\eps'$ judiciously to complete the proof of the first two claims. For the final claim, \eqref{eq:weak count}, the only adjustment to the above is to note that with all the random variables defined in analogous ways (using the same $\delta>0$ for the partition)\begin{equation*}
	 	p^{\pm}=(1+o(1))\frac{n^{-(2\log n)^{-\eps''/2}}}{\sqrt{2\pi}(2\log n)^{\frac{1}{2}-\eps''}},
	 \end{equation*} and hence the expected value which $N^{\pm}$ concentrates around needs to be adjusted. After this the proof is identical.

	\section{Largest entries of $D$}  \label{sec:largest-entries}
	In this section we prove the following technical result which will be necessary to prove our main results. 
	
	\begin{proposition}\label{prop:convergence of diagonal process}
		 The point process 
		\[
		\mathcal{P}_n  = \sum_{i=1}^n \delta_{a_n ( L_{ii} - b'_n) }
		\]
		converges to a Poisson point process with intensity measure $\mu(\dd x) = e^{-x} \dd x$.
	\end{proposition}

	 Let $\{X_{ij}\}$ be the triangular array such that $a_{ij}=\frac{1}{\sqrt{n}}X_{ij}$, let $Y_i = \frac{1}{\sqrt{n}} \sum_{j \neq i} X_{ij}=L_{ii}$, and let $u_n(x) = b'_n + \frac{x}{a_n}$.  Thus, 
	\[
	\P(Y_i > u_n(x)) = \P(a_n ( L_{ii} - b'_n) > x)
	\] The conclusion of Proposition \ref{prop:convergence of diagonal process} in fact holds even if we assume the entries of the Wigner matrix $A_n$ are only \emph{sub-exponential}, and we will complete the proof only under this assumption. Before beginning the proof, we state a corollary, which is actually equivalent to Proposition \ref{prop:convergence of diagonal process}, on the order statistics of the diagonal. 
	
	\begin{corollary}\label{cor:diag order}
		Let $L_{(1)}\geq L_{(2)}\geq\cdots\geq L_{(n)}$ be the order statistics of $\{L_{11},\dots,L_{nn}\}$, and let $k$ be a fixed positive integer. Then, \begin{equation}
			\lim\limits_{n\rightarrow\infty}\P\left(a_n(L_{(1)}-b_n')\leq x_1,\dots, a_n(L_{(k)}-b_n')\leq x_k \right)=F_{k}(x_1,\dots, x_k),
		\end{equation} where $F_k$ is defined in \eqref{eq:Gumbel limit stats joint}.
	\end{corollary}

	\begin{lemma}\label{lem:One row sum converges}
		For every compact set $K \subset \R$,
		\[
		\sup_{x \in K} |n \P(Y_1 > u_n(x)) - e^{-x}| \rightarrow 0
		\]
	\end{lemma}
	\begin{proof}
		The proof follows in a straightforward way from Lemmas \ref{lem:petrov's lemma} and \ref{lem:Gaussian tail max idk}. 
	\end{proof}
	
	Next, we establish that that the row sums are asymptotically independent in a strong enough sense.
	\begin{proposition}\label{prop:asympt ind of row sums}
		Fix $r \geq 1$ and $x_1, \dots, x_r \in \R$.  Then,
		\[
		\P(Y_1 > u_n(x_1), \dots, Y_r > u_n(x_r)) = \prod_{p = 1}^r \P(Y_1 > u_n(x_p)) (1 + o(1)).
		\]
		Moreover,
		\[
		n^r \P(Y_1 > u_n(x_1), \dots, Y_r > u_n(x_r)) \rightarrow e^{-(x_1 + \cdots + x_r)},
		\] uniformly on compact subsets of $\R^r$. 
	\end{proposition}
	
	\begin{proof}
		The case $r = 1$ is the content of Lemma \ref{lem:One row sum converges}.  Let $r \geq 2$ and write
		\[
		Y_p = U_p + V_p
		\]
		where 
		\[
		U_p = \frac{1}{\sqrt{n}} \sum_{j = r+1}^{n} X_{pj},
		\]
		and
		\[
		V_p = \frac{1}{\sqrt{n}} \sum_{q \leq r} X_{pq}.
		\]
		Note that the $U_1,U_2,\dots,U_r$ are independent.  Additionally, the collections $\{U_p\}$ and $\{V_p\}$ are independent.  Clearly, each $U_p$ has the same distribution as
		\[
		\sqrt{\frac{n-r}{n}} \frac{S_{n-r}}{\sqrt{n-r}},
		\] where $S_k$ is the sum of $n$ independent copies of $X_{12}$. 
		Now let $\theta>0$ be a constant such that \begin{equation}
			\E e^{\theta|X_{12}|}<\infty,
		\end{equation} and choose $K$ large enough so that $\theta K > r + 1$.  Let 
		\[
		\mathcal{E}_n := \left\{\max_{1 \leq p < q \leq r} |X_{pq}| \leq K \log n\right\}.
		\]
		By our assumption on the random variables, for $t \geq 0$,
		\[
		\P(|X| > t) \leq  C e^{-\theta t}.
		\]
		By a union bound,
		\[
		\P(\mathcal{E}_n^c) \leq \binom{r}{2} C n^{-\theta K} = o(n^{-r}).
		\]
		On the event $\mathcal{E}_n$,
		\[
		\max_{1 \leq p \leq r} |V_p| \leq \frac{(r-1) K \log n}{\sqrt{n}},
		\]
		and so
		\[
		a_n \max_p |V_{p}| \rightarrow 0.  
		\]
		For $y \in \R$ such that $|y| \leq (r-1) K \log n / \sqrt{n}$, let 
		\[
		q_p(y) = \P(U_p > u_n(x_p) - y).
		\]
		Since,
		\[
		\frac{u_n(x_p) - y}{\sqrt{(n-r)/n}} = u_n(x_p) + O(\log n/\sqrt{n}) = u_n(x_p) + o(a_n^{-1}),
		\]
		Lemmas \ref{lem:Gaussian tail max idk} and \ref{lem:One row sum converges} imply  that 
		\[
		q_p(y) = \frac{e^{-x_p} + o(1)}{n},
		\]uniformly in $y$ on any compact set. 
		By the Corollary,
		\begin{equation} \label{estimate}
			\sup_{|y| \leq (r-1) K \log n / \sqrt{n}} \left| \frac{q_p(y)}{\P(Y_1 > u_n(x_p)) } -1 \right| \rightarrow 0.  
		\end{equation}
		Now, if we condition on $V_p$, by the independence of $U$ from $V$,
		\[
		\P(Y_1 > u_n(x_1), \dots, Y_r > u_n(x_r) \cap \mathcal{E}_n) = \E \left[ \mathrm{1}_{\mathcal{E}_n} \prod_{p=1}^r q_p(V_p) \right].
		\]
		On $\mathcal{E}_n$, by \eqref{estimate}, 
		\[
		\E \left[ \mathrm{1}_{\mathcal{E}_n} \prod_{p=1}^r q_p(V_p) \right] = \prod_{p=1}^r \P(Y_1 > u_n(x_p)) (1 + o(1)).  
		\]
		Now, since $\P(\mathcal{E}_n^c) = o(n^{-r})$ and $\prod_{p=1}^r \P(Y_1 > u_n(x_p)) = \Theta( n^{-r})$, the result is proved.
	\end{proof}
	
	\begin{lemma}\label{lem:counting prob of diag}
		For $1 \leq p \leq r$,
		let $I_p = (\alpha_p, \beta_p]$.  Set $T_i = a_n(Y_i - b'_n) $
		and
		\[
		\mu(I_p) = e^{-\alpha_p} - e^{-\beta_p}.  
		\]
		Then,
		\[
		\P(T_1 \in I_1, \dots, T_r \in I_r) = n^{-r} \prod_{p=1}^r  \mu(I_p) + o(n^{-r})
		\]
	\end{lemma}
	\begin{proof}
		We use inclusion-exclusion. 
		\[
		\P(T_1 \in I_1, \dots, T_r \in I_r) = n^{-r} \sum_{B \subset \{1, \dots, r\}} (-1)^{|B|}\prod e^{-z_p(B)} + o(n^{-r}) = n^{-r} \prod (e^{-\alpha_p} - e^{-\beta_p}) + o(n^{-r})
		\] 
		where $z_p(B) = \beta_p$ if $p \in B$, otherwise, $z_p(B) = \alpha_p$.  
	\end{proof}
	\subsection{Proof of Proposition \ref{prop:convergence of diagonal process}}
	
	\begin{proof}
		Let $I_1,I_2,\dots, I_m$ be disjoint intervals of the form $I_i=(x_i,y_i]$. Let $N_{n,\ell} = \mathcal{P}_n(I_\ell)$.  Fix $k_1, \dots, k_m \in \mathbb{N}$ and $r = k_1 + \dots + k_m$.  For $r \geq 1$,
		
		\[
		\E \prod_{\ell=1}^m (N_{n,\ell})_{k_\ell} = (n)_r \P(T_1, \dots, T_{k_1} \in I_1, T_{ k_1 + 1}, \dots, T_{ k_1+k_2} \in I_2, \cdots, T_{k_{m-1}+1 },\dots, T_{k_m}\in I_{m}), 
		\]
		where $(z)_k = z(z-1) \cdots (z - k+1)$. Lemma \ref{lem:counting prob of diag} gives
		\[
		\E \prod_{\ell=1}^m (N_{n,\ell})_{k_\ell} \rightarrow \prod_{\ell=1}^m \mu(I_\ell)^{k_\ell}.  
		\]
		The joint law is determined by these moments, see \cite[Theorem 1.23]{Bollobas2001}. Thus
		$(N_{n,1}, \dots, N_{n,m})$ converges to a vector of independent Poisson random variables with means $e^{-x_i}-e^{-y_i}$. Convergence of these finite dimensional distributions implies convergence of the point process, completing the proof. 		
	\end{proof}

	\section{Proof of Theorem \ref{thm:largest k eigenvalues}} \label{sec:proof}
	We define the events 
	\[
	\mathcal{E}_1 = \left\{ \left| \{ 1 \leq i \leq n: D_{ii} \geq \sqrt{(2 - 2\delta) \log n} \} \right| \leq n^{\varepsilon} \right\} 
	\]
	\[
	\mathcal{E}_2 = \bigcap_{k} \left\{ \left| \{ 1 \leq i \leq n: \tilde{D}^{(k)}_{ii} \geq \sqrt{(2 - 2\delta) \log n} \} \right| \leq n^{\varepsilon}   \right\}
	\]
	\[
	\mathcal{E}_3 = \bigcap_{i,j \in [n]} \{ |a_{ij}| \leq n^{- 1/2 + \varepsilon}\}
	\]
	Next, we define	$\mathcal{E} := \mathcal{E}_1 \cap \mathcal{E}_2 \cap \mathcal{E}_3\cap\{\|A\|\leq 3 \}$, which holds with overwhelming probability for some choices of $\delta,\eps>0$ due to Proposition \ref{prop:eigen count}, the fact that the entries of $A$ are sub-Gaussian, and well established bounds on the operator norm of Wigner matrices \cite{MR4550215}.  Throughout the proof we will need to update the event $\mathcal{E}$ by taking the intersection with other events which holds with overwhelming probability. To avoid cumbersome notation, we will point out each update to $\mathcal{E}$ and simply denote this new event  again by $\mathcal{E}$.

			By the Schur Complement Formula (Lemma \ref{lem:schur}),
		\begin{align} \label{eq:schurexpansion}
			G_{kk}(z)  &= \frac{1}{ \sum_{i}^{(k)} a_{ik} - z- \sum_{i,j}^{(k)} a_{ki} G_{ij}^{(k)} a_{jk}} \nonumber \\ 
			&=  \frac{1}{ \sum_{i}^{(k)} a_{ik} - z- m_n(z) - r_k} 
		\end{align}
	where
	\[
	r_k = \sum_{ij} G_{ij}^{(k)} \left(a_{ki} a_{kj} - \delta_{ij}\frac{1}{n} \right) - (m_n(z) - m_n^{(k)}(z)).
	\]

	We prove initial estimates for $m_n$, $m_n^{(k)}$, $\tilde{m}$, and $r_k$ which will  later be improved.
	
	\begin{lemma}\label{lem:weak local law} On $\mathcal{E}$,
		\begin{equation}
			 \sup_{z\in S_{\delta}}|\tilde{m}_n^{(k)}(z)|=o(1),
		\end{equation} and   \begin{equation}
		 \sup_{z\in S_{\delta}}|\Im \tilde{m}_{n}^{(k)}(z)|=O\left(\frac{\eta}{\log n}\right),
		\end{equation} for any $k\in[n]$. Additionally, the same results hold for both $m_n$ and $m_n^{(k)}$.  
	\end{lemma}
	
	\begin{proof}
		We will prove the results only for $\tilde{m}_n^{(k)}$. The proofs for $m_n$ and $m_n^{(k)}$ are identical. 
		
		Partition $[n]$ into two sets \begin{equation}
			\begin{aligned}
				J_1&=\{1 \leq i \leq n-1: \tilde{\lambda}^{(k)}_{i} < \sqrt{(2 -2 \delta) \log n}\},\text{ and }\\
				J_2&=\{1 \leq i \leq n-1: \tilde{\lambda}^{(k)}_{i} \geq\sqrt{(2 - 2\delta) \log n}\}.
			\end{aligned}
		\end{equation} From the triangle inequality and Cauchy's interlacing inequalities, $|J_{2}|=O_{\eps}(n^{\eps})$ on $\mathcal{E}$.  Then, for any $z\in S_{\delta}$
        \begin{equation}
		\begin{aligned}
			\oindicator{\mathcal{E}} |\tilde{m}_{n}^{(k)}(z)|&\leq \oindicator{\mathcal{E}} \frac{1}{n} \sum_{j=1}^{n-1}\left|\frac{1}{({\lambda}_{j}(\tilde{L}^{(k)})-E)+\ii\eta}\right|\\
			&\leq\oindicator{\mathcal{E}} \frac{1}{n}\sum_{j\in J_1}\left|\frac{1}{({\lambda}_{j}(\tilde{L}^{(k)})-E)+\ii\eta}\right|+\oindicator{\mathcal{E}} \frac{1}{n}\sum_{j\in J_2}^{n-1}\left|\frac{1}{({\lambda}_{j}(\tilde{L}^{(k)})-E)+\ii\eta}\right|\\
			&=O\left(\frac{1}{\sqrt{\log n}} \right) + O\left(\frac{n^{\eps}}{n\eta}\right)\\
			&=o(1).
		\end{aligned}
		\end{equation} Similarly, for any $z\in S_{\delta}$ \begin{equation}
		\begin{aligned}
			\oindicator{\mathcal{E}} \eta^{-1}\Im \tilde{m}_{n}^{(k)}(z)&=\oindicator{\mathcal{E}} \frac{1}{n}\sum_{j=1}^{n-1}\frac{1}{({\lambda}_{j}(\tilde{L}^{(k)})-E)^2+\eta^2}\\\
			&\leq\oindicator{\mathcal{E}} \frac{1}{n}\sum_{j\in J_1}\frac{1}{({\lambda}_{j}(\tilde{L}^{(k)})-E)^2+\eta^2}+\oindicator{\mathcal{E}} \frac{1}{n}\sum_{j\in J_2}^{n-1}\frac{1}{({\lambda}_{j}(\tilde{L}^{(k)})-E)^2+\eta^2}\\
			&=O\left(\frac{1}{{\log n}} \right)+O\left(\frac{n^{\eps}}{n\eta^2}\right)\\
			&=o(1).
		\end{aligned}
		\end{equation}
	\end{proof}
	
	Next, we compare $m^{(k)}$ and $\tilde{m}^{(k)}$ explicitly, as it will be helpful to be able to exchange the two in what follows. 
	\begin{lemma}\label{lem:comp m and tilde m}
    One has
		 \begin{equation}\label{eq:comp m and tilde m}
			\sup_{z\in S_{\delta}}\left|m_{n}^{(k)}(z)-\tilde{m}_n^{(k)}(z)\right|\lesssim \sqrt{\frac{1}{n\log n}},
		\end{equation} and \begin{equation}\label{eq:comp Im m and tilde m}
		\sup_{z\in S_{\delta}}\left|\Im m_{n}^{(k)}(z)-\Im\tilde{m}_n^{(k)}(z)\right|\lesssim n^{\eps}\eta^3,
		\end{equation} for all $k$ with overwhelming probability, where $\eps>0$ is the parameter in the definition of the event $\mathcal{E}$.
	\end{lemma}
	
	\begin{proof}
		Let $\pi$ be any permutation of $[n-1]$  and define the sets \begin{equation}
			\begin{aligned}
				J_1&=\{1 \leq i \leq n-1: \tilde{\lambda}^{(k)}_{\pi(i)} < \sqrt{(2 -2 \delta) \log n}\text{ and }  {\lambda}^{(k)}_{i} < \sqrt{(2 -2 \delta) \log n} \},\text{ and }\\
				J_2&=[n-1]\setminus J_1.
			\end{aligned}
		\end{equation} It follows from the union bound that \begin{equation}
		|J_2|\lesssim n^{\eps}
		\end{equation} on $\mathcal{E}$. From the triangle and Jensen's inequality \begin{equation}
		\begin{aligned}
			\sup_{z\in S_{\delta}}\left|m^{(k)}_n(z)-\tilde{m}_n^{(k)}(z)\right|^2&=\sup_{z\in S_{\delta}}\left|\frac{1}{n-1}\sum_{i=1}^{n-1}\frac{1}{\lambda_i^{(k)}-z }-\frac{1}{\tilde\lambda_{\pi(i)}^{(k)}-z }\right|^2\\
			&\leq \sup_{z\in S_{\delta}}\left|\frac{1}{n-1}\sum_{i\in J_1}\frac{1}{\lambda_i^{(k)}-z }-\frac{1}{\tilde\lambda_{\pi(i)}^{(k)}-z }\right|^2\\
			&\qquad+\sup_{z\in S_{\delta}}\left|\frac{1}{n-1}\sum_{i\in J_2}^{}\frac{1}{\lambda_i^{(k)}-z }-\frac{1}{\tilde\lambda_{\pi(i)}^{(k)}-z }\right|^2\\
			&\lesssim \frac{1}{(n-1)\log^4 n}\sum_{i\in J_1}|\lambda_i^{(k)}-\tilde{\lambda}_{\pi(i)}^{(k)}|^2+\frac{n^{2\eps}}{(n-1)^2\eta^2}.
		\end{aligned}
		\end{equation} From \cite[Theorem A.37]{MR2567175}, \begin{equation}
		\min_{\pi}\sum_{i=1}^{n-1}|\lambda_i^{(k)}-\tilde{\lambda}_{\pi(i)}^{(k)}|^2\leq\tr(L^{(k)}-\tilde{L}^{(k)} )^2=\sum_{i}^{(k)}a_{ik}^2=O\left(\log n\right),
		\end{equation} with overwhelming probability, where the minimum is over all permutations of $[n-1]$. Thus, after an appropriate choice of $\pi$,  \begin{equation}
			\sup_{z\in S_{\delta}}\left|m^{(k)}_n(z)-\tilde{m}_n^{(k)}(z)\right|^2\lesssim\frac{1}{(n-1)\log^3n}+\frac{n^{2\eps}}{(n-1)^2\eta^2},
		\end{equation} with overwhelming probability, completing the proof of \eqref{eq:comp m and tilde m}. The same argument also yields \begin{equation}
		\sup_{z\in S_{\delta}}\left|\Im m_{n}^{(k)}(z)-\Im\tilde{m}_n^{(k)}(z)\right|^2\lesssim \frac{\eta^2}{(n-1)\log^3n}+\frac{n^{2\eps}}{(n-1)^2\eta^2},
		\end{equation} with overwhelming probability. 
	\end{proof}
	
	We again update $\mathcal{E}$ to include the intersection with the event in \eqref{eq:comp m and tilde m}.
	\begin{lemma} \label{lem:r_k}
		For all $k \in [n]$ and all $\xi >0$, with overwhelming probability,
		\[
		\sup_{z\in S_{\delta}}	|r_k| = O(n^{-1/2} \log^{2 \xi}n).
		\]
	\end{lemma}	
	\begin{proof}
	\begin{align*}
	r_k &= \sum_{ij} G_{ij}^{(k)} \left(a_{ki} a_{kj} - \delta_{ij}\frac{1}{n} \right) - (m_n(z) - m_n^{(k)}(z)) \\
	&= \sum_{i} G_{ii}^{(k)} \left(a_{ki}^2 - \frac{1}{n}\right) + \sum_{i \neq j} G_{ij}^{(k)} a_{ki} a_{kj} - (m_n(z) - m_n^{(k)}(z)).
	\end{align*}
 	The main difficulty is that that $G_{ij}^{(k)}$ and $a_{ki}$ are dependent.  We use the resolvent identity to decouple these terms.  Observe that for any $m \in \mathbb{N}$, 
 	\[
 	G = \tilde{G} + \tilde{G} \Delta \tilde{G} + \tilde{G} \Delta \tilde{G} \Delta G,
 	\]
 	where $\Delta = \tilde{L}^{(k)} - L^{(k)}$.  Note that $\Delta$ is a diagonal matrix with $-\mathbf{a}^{(k)}_k$ along the diagonal.  Thus,
 	\begin{align*}
 	r_k &= \left((\mathbf{a}^{(k)})^\mathrm{T} \tilde{G}^{(k)} \mathbf{a}^{(k)} - m_n^{(k)}(z) \right) + (\mathbf{a}^{(k)})^\mathrm{T} \tilde{G}^{(k)} \Delta_k \tilde{G}^{(k)} \mathbf{a}^{(k)} \\
 	&\qquad + (\mathbf{a}^{(k)})^\mathrm{T} \tilde{G}^{(k)} \Delta_k \tilde{G}^{(k)} \Delta_k G^{(k)} \mathbf{a}^{(k)} - (m_n(z) - m_n^{(k)}(z)).
 	\end{align*}
	
	We first observe that 
	\begin{align} \label{eq:tracebound}
	\oindicator{\mathcal{E}} \sum_{i} |\tilde{G}^{(k)}_{ii}|^2 &\leq \oindicator{\mathcal{E}} \sum_{i,j} |\tilde{G}^{(k)}_{ij}|^2 \nonumber \\
	&\leq \oindicator{\mathcal{E}} \tr \tilde{G}^{(k)} (\tilde{G}^{(k)})^* \nonumber \\
	&= \oindicator{\mathcal{E}} \sum_i \frac{1}{(\lambda_i(\tilde{L}^{(k)}) - E)^2 + \eta^2} \\
	&\lesssim \frac{n}{\log n} + n^{\varepsilon} \eta^{-2},
	\end{align}
	where the last line follows from the definition of the event $\mathcal{E}$.  
	By Lemma \ref{lem:HW} and \ref{lem:comp m and tilde m}, with overwhelming probability, for all $k$, and any $\xi > 0$,
	\[
	\left|(\mathbf{a}^{(k)})^\mathrm{T} \tilde{G}^{(k)} \mathbf{a}^{(k)} - {m}_n^{(k)}(z)  \right| \leq 	\left|(\mathbf{a}^{(k)})^\mathrm{T} \tilde{G}^{(k)} \mathbf{a}^{(k)} - \tilde{m}_n^{(k)}(z)  \right| +	\left|\tilde{m}_n^{(k)}(z)- {m}_n^{(k)}(z)  \right| \lesssim n^{-1/2} \log^{\xi} n. 
	\]

	Using Lemma \ref{lem:concentration of vec-matrix} we have that 
	\begin{equation}\label{eq:subGauss bound on aG}
		\left\| \|\mathbf{a}^{T}\tilde{G}^{(k)}\|_{2}-\|\tilde{G}^{(k)}\|_{HS}/\sqrt{n}    \right\|_{\psi_2}\leq \frac{C\|\tilde{G}^{(k)}\|}{\sqrt{n}}=O\left(\frac{1}{\eta\sqrt{n}}\right),
	\end{equation}where the sub-Gaussian norm is taken conditionally on $\tilde{G}^{(k)}$. Then, using the Ward identity and Lemma \ref{lem:weak local law} we have with overwhelming probability \begin{equation}\label{eq:HS bound on G}
		\begin{aligned}
			\|\tilde{G}^{(k)} \|_{HS}/\sqrt{n}&=\frac{1}{\sqrt{n}}\sqrt{\sum_{i,j}^{(k)}|\tilde{G}^{(k)}_{ij}|^{2}    }\\
			&=\frac{1}{\sqrt{n}}\sqrt{\eta^{-1}\sum_{i}^{(k)}\Im \tilde{G}^{(k)}_{ii}    }\\
			&=\sqrt{\eta^{-1}\Im\tilde{m}_n^{(k)}(z)}\\
			&=O\left(\frac{1}{\log n}\right),
		\end{aligned}
	\end{equation} on $\mathcal{E}$.  Combining \eqref{eq:subGauss bound on aG} and \eqref{eq:HS bound on G}, \begin{equation}
	\|(\mathbf{a}^{(k)})^\mathrm{T} \tilde{G}^{(k)}\|_2 =O\left(\frac{1}{\log n}\right),
	\end{equation} with overwhelming probability.

	Since $\Delta$ is diagonal with sub-Gaussian entries, we also have that  
	\[
	\|\Delta\| \lesssim \sqrt{\frac{\log n}{n}},
	\] with overwhelming probability.

	Combining these estimates with the deterministic bound $\|G^{(k)}\| \leq \eta^{-1}$, we conclude that with overwhelming probability,
	\begin{align*}
	|(\mathbf{a}^{(k)})^\mathrm{T} \tilde{G}^{(k)} \Delta_k \tilde{G}^{(k)} \mathbf{a}^{(k)} | &\leq \|(\mathbf{a}^{(k)})^\mathrm{T} \tilde{G}^{(k)}\|_2 \|\Delta \| \|\tilde{G}^{(k)} \mathbf{a}^{(k)}\|_2 \\
	&\lesssim \sqrt{\frac{1 }{n\log n}},\\
	\end{align*}	
	and
	\begin{align*}
	|(\mathbf{a}^{(k)})^\mathrm{T} \tilde{G}^{(k)} \Delta_k \tilde{G}^{(k)} \Delta_k G^{(k)} \mathbf{a}^{(k)}| &\lesssim  \frac{\log^2 n}{n^{3/2}} \eta^{-2} \\
	&\lesssim \frac{\log n}{n}.
	\end{align*} 
	Combining these estimates with Lemma \ref{lem:rank1} and noting the above is uniform over $S_{\delta}$ concludes the proof of the result for fixed $z\in S_{\delta}$.  
	To extend from fixed $z$ to $\sup_{z\in S_{\delta}}$ we observe that $r_{k}$ depends on $z$ only through $m_n, m_n^{(k)},$ and a linear combination of entries of $G$. As functions on $S_{\delta}$, all of these terms are Lipschitz in $z$ with Lipschitz constants bounded by $\eta^{-2}$. It then follows that on $\mathcal{E}$, $r_{k}$ as a function over $S_{\delta}$ is Lipschitz in $z$ with Lipschitz constant bounded by $n^{2+\eps}$. Let $A\subset S_{\delta}$ be an $n^{-10}$-net, i.e.\ a subset of $S_{\delta}$ such that any element of $S_{\delta}$ is at most $n^{-10}$ away from a element of $A$, of cardinality $|A|\leq n^{21}$. We then show \begin{equation*}
		\sup_{z\in A} |r_{k}|\leq O\left(n^{-1/2}\log^{2\xi}n\right),
	\end{equation*}  with overwhelming probability using a simple union bound, and then extend this to 
		\[
	\sup_{z\in S_{\delta}}	|r_k| = O(n^{-1/2} \log^{2 \xi}n),
	\] using the Lipschitz bound for $r_{k}$. 
	\end{proof}

	At this point we update the event $\mathcal{E}$ to include the intersection with \begin{equation}
		\bigcap_{k} \left\{\sup_{z\in S_{\delta}} |r_k| \leq(n^{-1/2} \log^{2 \xi}n) \right\},
	\end{equation} for some $\xi>0$. We will need a slight refinement of our previous result to the imaginary component of the error term.
	
	\begin{lemma} \label{cor:imr}
	For all $k\in [n]$,	\[
		\sup_{z\in S_{\delta}}|\Im r_k(z)| \lesssim \eta^3 n^{\eps},
		\] with overwhelming probability, where $\eps>0$ is the parameter in the definition of the event $\mathcal{E}$.
		
	\end{lemma}
	\begin{proof}
		By \eqref{eq:comp Im m and tilde m}, we may replace the $\Im m_n^{(k)}$ with $\Im \tilde{m}_{n}^{(k)}$, up to a loss of $n^{\eps}$. By Lemma \ref{lem:HW}, on the event $\mathcal{E}$
		\begin{align}
		 \sum_{ij}^{(k)}  \Im \tilde{G}_{ij}^{(k)} (a_{ki} a_{kj} - \delta_{ij} \frac{1}{n}) &\leq \left(\frac{1}{n^2} \sum_{ij} |\Im \tilde{G}^{(k)}_{ij}|^2 \right)^{1/2} \log^{1/2+\xi} n\nonumber \\
		 &= \frac{1}{n}\left( \tr \frac{\tilde{G}^{(k)} - (\tilde{G}^{k})^*}{2\ii} \frac{\tilde{G}^{(k)} - (\tilde{G}^{(k)})^*}{2\ii} \right)^{1/2} \log^{1/2+\xi} n \nonumber\\
		 &\leq \frac{1}{n} \left( \sum \frac{\eta^2}{((\lambda_i(\tilde{L}^{(k)}) - E)^2 + \eta^2)^2}\right)^{1/2} \log^{1/2+\xi} n\nonumber \\
		 &\lesssim \frac{1}{n}\left(\frac{n\eta^2}{\log n}+n^{\eps}\eta^{-2} \right)^{1/2}\log^{\frac{1}{2}+\xi}n,\label{eq:first im r_k error}
		\end{align} with overwhelming probability. 
	
		For the next term, we observe that
		\begin{align*}
		|\Im (\mathbf{a}^{(k)})^\mathrm{T} \tilde{G}^{(k)} \Delta_k \tilde{G}^{(k)} \mathbf{a}^{(k)}| &\leq |(\mathbf{a}^{(k)})^\mathrm{T} \Im \tilde{G}^{(k)} \Delta_k \Re \tilde{G}^{(k)} \mathbf{a}^{(k)}| + |(\mathbf{a}^{(k)})^\mathrm{T} \Re \tilde{G}^{(k)} \Delta_k \Im \tilde{G}^{(k)} \mathbf{a}^{(k)}| \\
		&\leq 2 \|(\mathbf{a}^{(k)})^\mathrm{T} \Im \tilde{G}^{(k)}\|_2 \|\Delta_k\| \|\Re \tilde{G}^{(k)} \mathbf{a}^{(k)}\|_2. 
		\end{align*} From here, applying Lemma \ref{lem:concentration of vec-matrix} to $\|(\mathbf{a}^{(k)})^{\mathrm{T}}\Im\tilde{G}^{(k)}\|_2$ gives the bound\begin{equation}\label{eq:a Im G HS bound}
		\begin{aligned}
			\|(\mathbf{a}^{(k)})^{\mathrm{T}}\Im\tilde{G}^{(k)}\|_2&\lesssim \left\|\Im \tilde{G}^{(k)} \right\|_{HS}/\sqrt{n}\\
			&\lesssim \left( \frac{1}{n}\sum_{i} \frac{\eta^2}{((\lambda_i(\tilde{L}^{(k)}) - E)^2 + \eta^2)^2}\right)^{1/2}\\
			&\lesssim \eta n^{\eps/2},
		\end{aligned}
		\end{equation} and
		\begin{equation}\label{eq:Im rk biggest loss}
		|\Im(\mathbf{a}^{(k)})^\mathrm{T} \tilde{G}^{(k)} \Delta_k \tilde{G}^{(k)}  \mathbf{a}^{(k)}|\lesssim \eta n^{-1/2}n^{\eps/2}\log^{-1/2}n,
		\end{equation} with overwhelming probability.

		For the term of $r_k$ involving three resolvents, we need to apply the resolvent formula one more time.
		\begin{align*}
		|(\mathbf{a}^{(k)})^\mathrm{T} \tilde{G}^{(k)} \Delta_k \tilde{G}^{(k)} \Delta_k G^{(k)} \mathbf{a}^{(k)}| &\leq  	\|(\mathbf{a}^{(k)})^\mathrm{T} \tilde{G}^{(k)}\|_2 \|\Delta_k\| \|\tilde{G}^{(k)}\| \|\Delta_k\| \|\tilde{G}^{(k)} \mathbf{a}^{(k)}\|_2 \\
		&\quad + \|(\mathbf{a}^{(k)})^\mathrm{T} \tilde{G}^{(k)}\|_2 \|\Delta_k\| \|\tilde{G}^{(k)}\| \|\Delta_k\| \|\tilde{G}^{(k)}\| \|\Delta_k\| \|\tilde{G}^{(k)} \mathbf{a}^{(k)}\| \\
		&\quad + \|(\mathbf{a}^{(k)})^\mathrm{T} \tilde{G}^{(k)}\|_2 \|\Delta_k\| \|\tilde{G}^{(k)}\| \|\Delta_k\| \|\tilde{G}^{(k)}\| \|\Delta_k\| \|\tilde{G}^{(k)}\| \|\Delta_k\| \|{G}^{(k)}\| \|\mathbf{a}^{(k)}\| \\
		&\lesssim \log^{3/2} n \eta^{3},
		\end{align*} with overwhelming probability. 
	\end{proof} We again update the event $\mathcal{E}$ to include the intersection with \begin{equation*}
	\bigcap_{k}\left\{\sup_{z\in S_{\delta}}|\Im r_k(z)| \leq \eta^3 n^{\eps/2}\log^{2 \xi} n \right\},
	\end{equation*} for some $\xi>0$.

	For the remainder of the proof we will only work with the eigenvalues $\lambda_{j}\equiv \lambda_{j}(L)$ of the full Laplacian $L$.  Now, we let $z_1 = E + \ii \eta$ and $z_2 = E+ \ii \sqrt{2} \eta$ and define 
	\begin{align*}
	I_1 &:= n \eta  (\Im m_n(z_1) - \frac{1}{\sqrt 2} \Im m_n(z_2))\\
	&= \sum_{j=1}^n \frac{\eta^2}{(\lambda_j - E)^2 + \eta^2} - \sum_{j=1}^n \frac{\eta^2}{(\lambda_j - E)^2 + 2 \eta^2} \\
	&= \eta \sum_j \Im(G_{jj} (z_1)  - \frac{1}{\sqrt 2} G_{jj}(z_2)) 
	\end{align*}
	and
	\begin{align*}
	I_2 &:=  \eta \Im \sum_{k=1}^n \frac{1}{ \sum_{i}^{(k)} a_{ik} - z- m_n(z_1) } -  \frac{1/\sqrt 2}{ \sum_{i}^{(k)} a_{ik} - z- m_n(z_2) } \\
	&=\sum_{k=1}^n \frac{\eta^2+ \eta \Im m_n(z_1)}{(\sum_i^{(k)} a_{ik} - E - \Re m_n(z_1))^2 +(\eta + \Im m_n(z_1))^2}  \\
	&\qquad - \sum_{j=1}^n \frac{\eta^2+ \frac{1}{\sqrt{2}}\eta \Im m_n(z_2)}{(\sum_i^{(k)} a_{ik} - E - \Re m_n(z_2))^2 +(\sqrt{2} \eta + \Im m_n(z_2))^2}
	\end{align*}
	
	The reason we consider $I_1$ and $I_2$ is that these differences of imaginary terms  will essentially be a factor of $\eta^2$ smaller than the individual terms. We need one final estimate for the error $r_{k}$ at the two different points $z_1$ and $z_2$, which uses this idea of gaining extra factors of $\eta$ by considering differences of imaginary parts. \begin{lemma}\label{lem:diff Im r}
		For all $k\in[n]$, \begin{equation}\label{eq:diff im r}
			\sup_{z_1\in S_{\delta}}\left|\eta\Im r_k(z_1)-\frac{\eta}{\sqrt{2}}\Im r_k(z_2)\right| \lesssim \eta^4,
		\end{equation}with overwhelming probability.
	\end{lemma}
	\begin{proof}
		We first note that an estimate better than \eqref{eq:diff im r} is likely to hold. We will prove only this relatively weak statement to avoid technical details. The proof is mostly identical to the proof of Lemmas \ref{lem:r_k} and \ref{cor:imr}. We note that some steps in the proof of Lemma \ref{cor:imr} actually prove the desired bound for the individual terms after multiplying by the additional fact of $\eta$, so by the triangle inequality it is sufficient to use the difference of imaginary parts to improve the bounds on the terms analogous to \eqref{eq:comp Im m and tilde m}, \eqref{eq:first im r_k error}, \eqref{eq:a Im G HS bound}, and \eqref{eq:Im rk biggest loss}. All these steps are essentially identical, so we will focus on the more technical one in which the terms of		
		$r_k$ involve  $\tilde{G}^{(k)}\Delta_k\tilde{G}^{(k)}$, which in this case of a difference of imaginary parts becomes \begin{equation}
			\begin{aligned}
				&\Im \eta^2(\mathbf{a}^{(k)})^\mathrm{T}  \tilde{G}^{(k)}(z_1) \Delta_k  \tilde{G}^{(k)}(z_1) \mathbf{a}^{(k)}-\Im \frac{\eta^2}{2}(\mathbf{a}^{(k)})^\mathrm{T}  \tilde{G}^{(k)}(z_2) \Delta_k  \tilde{G}^{(k)}(z_2) \mathbf{a}^{(k)}\\
				&=\Im (\mathbf{a}^{(k)})^\mathrm{T}\left(\eta \tilde G(z_1)-\frac{\eta}{\sqrt{2}}  \tilde G(z_2) \right) \Delta_k \eta \tilde{G}(z_1)\mathbf{a}^{(k)}\\
				&\quad +\Im (\mathbf{a}^{(k)})^\mathrm{T} \frac{\eta}{\sqrt{2}}\tilde{G}(z_1) \Delta_k \left(\eta \tilde G(z_1)-\frac{\eta}{\sqrt{2}} \tilde G(z_2) \right)\mathbf{a}^{(k)}.
			\end{aligned}
		\end{equation} Proceeding exactly as in the proof of Lemma \ref{cor:imr} we show that \begin{equation}\label{eq:diff im r goal}
			\left\|\left(\mathbf{a}^{(k)}\right)^{\mathrm{T}}\left(\eta\Im \tilde G(z_1)-\frac{\eta}{\sqrt{2}}\Im \tilde G(z_2) \right) \right\|_2 \lesssim \eta^2 n^{\eps/2},
		\end{equation} with overwhelming probability. Applying Lemma \ref{lem:concentration of vec-matrix}, \eqref{eq:diff im r goal} follows by noting the improved Hilbert--Schmidt norm of this matrix is very small, namely using Proposition \ref{prop:eigen count} again and noting that $\tilde{G}(z_1)$ and $\tilde{G}(z_2)$ are simultaneously diagonalizable 
        \begin{align*}
								&\frac{1}{\sqrt{n}}\left\|\eta\Im \tilde G(z_1)-\frac{\eta}{\sqrt{2}}\Im \tilde G(z_2) \right\|_{HS}\\&\quad=\frac{1}{\sqrt{n}}\left(\sum_j\left[\frac{\eta^2}{(\lambda_{j}(\tilde{L})-E)^2+\eta^2}-\frac{\eta^2}{(\lambda_{j}(\tilde{L})-E)^2+2\eta^2}\right]^2 \right)^{1/2}\\
								&\quad=\frac{1}{\sqrt{n}}\left(\sum_j\left[\frac{\eta^4}{[(\lambda_{j}(\tilde{L})-E)^2+\eta^2] [(\lambda_{j}(\tilde{L})-E)^2+2\eta^2]}\right]^2 \right)^{1/2}\\
								&\quad\lesssim \left(\frac{\eta^8}{\log^4 n}+\frac{n^{\eps}}{n}\right)^{1/2}\\
								&\quad\lesssim \eta^2n^{\eps/2},
		\end{align*}  with overwhelming probability. With this improved bound we conclude that \begin{equation}
		\left|\Im (\mathbf{a}^{(k)})^\mathrm{T}\left(\eta \tilde G(z_1)-\frac{\eta}{\sqrt{2}}  \tilde G(z_2) \right) \Delta_k \eta\tilde{G}(z_1)\mathbf{a}^{(k)}\right|\lesssim \eta^5\log^{2\xi}n^{\eps/2}\lesssim \eta^4.
		\end{equation} 
	\end{proof}

	Let 
	\[
	J_1 = \{i: D_{ii} < \sqrt{(2 - 2\delta) \log n} \}
	\]
	and
	\[
	J_2 = [n] \setminus J_1.
	\]
	On the event $\mathcal{E}$,
	\[
	\left|D_{kk} - \sum_i^{(k)} a_{ik} \right| = |a_{kk}| \leq n^{-1/2 + \varepsilon}. 
	\]
	
	We restructure the difference in the following way, 	\begin{align*}
		I_1 - I_2 &\leq \eta \sum_{k \in J_1} \Im(G_{kk} (z_1)  - \frac{1}{\sqrt 2} G_{kk}(z_2))    \\
		&\qquad + \sum_{k \in J_1} \frac{\eta^2+ \eta \Im m_n(z_1)}{(\sum_i^{(k)} a_{ik} - E - \Re m_n(z_1))^2 +(\eta + \Im m_n(z_1))^2}  \\
		&\qquad - \sum_{k \in J_1} \frac{\eta^2+ \frac{1}{\sqrt{2}}\eta \Im m_n(z_2)}{(\sum_i^{(k)} a_{ik} - E - \Re m_n(z_2))^2 +(\sqrt{2} \eta + \Im m_n(z_2))^2}  \\
		&\qquad +  \sum_{k \in J_2} \eta \Im(G_{kk}(z_1)) - \frac{\eta^2+ \eta \Im m_n(z_1)}{(\sum_i^{(k)} a_{ik} - E - \Re m_n(z_1))^2 +(\eta + \Im m_n(z_1))^2} \\
		&\qquad - \sum_{k \in J_2} \frac{\eta}{2} \Im(G_{kk}(z_2)) - \frac{\eta^2+ \frac{1}{\sqrt{2}}\eta \Im m_n(z_2)}{(\sum_i^{(k)} a_{ik} - E - \Re m_n(z_2))^2 +(\sqrt{2} \eta + \Im m_n(z_2))^2} 
	\end{align*} We control each term in the lemmas below. However, for the summations over $J_1$ we need to identify a few terms involving $I_1$ before we take the absolute value and apply the triangle inequality.


		\begin{lemma} \label{lem:imdifference}
		On $\mathcal{E}$, there exists a sequence $p_n\rightarrow0$ such that 
		\[
		 \left|p_nI_1- \eta\Im \sum_{k \in J_1} G_{kk}(z_1) - \frac{G_{kk}(z_2)}{\sqrt 2}\right| = o(1).
		\]
	\end{lemma}
	\begin{proof}
		By taking the imaginary part of \eqref{eq:schurexpansion}, we have that
		\begin{align*}
			\eta \Im G_{kk}(z_1) = \frac{\eta^2 + \eta \Im m_n(z_1) + \eta \Im r_k(z_1)}{(\sum_i^{(k)} a_{ik} - E - \Re m_n(z_1) - \Re r_k(z_1))^2 + (\eta + \Im m_n(z_1) + \Im r_k(z_1))^2}
		\end{align*}
		Thus, we define
		\[
		\eta \sum_{k\in J_1} \Im (G_{kk}(z_1) - \frac{G_{kk}(z_2)}{\sqrt 2} ) := \sum_{k \in J_1} \frac{N_{k}}{B_{k}}
		\]
		where
		\begin{align}
			N_{k} &:=   (\eta^2 + \eta \Im  r_k(z_1))(\sqrt{2} \eta + \Im m_n(z_2) + \Im r_k(z_2))^2 \nonumber \\
			&\qquad - (\eta^2 + \frac{\eta}{\sqrt{2}} \Im r_k(z_2))(\eta + \Im m_n(z_1) + \Im r_k(z_1))^2 \label{eq:Nk1}\\
			&\qquad + \eta \Im m_n(z_1)(\sqrt{2} \eta + \Im m_n(z_2) + \Im r_k(z_2))^2 \label{eq:Nk2}\\
			&\qquad - \frac{\eta}{\sqrt{2}} \Im m_n(z_2) (\eta + \Im m_n(z_1) + \Im r_k(z_1))^2 \label{eq:Nk3}\\
			&\qquad +  (\eta^2+ \eta \Im r_k(z_1))(\sum_i^{(k)} a_{ik} - E - \Re m_n(z_2) - \Re r_k(z_2))^2 \nonumber \\
			&\qquad - (\eta^2 + \frac{\eta}{\sqrt{2}}\Im r_k(z_2))(\sum_{i}^{(k)} a_{ik} - E - \Re m_n(z_1) - \Re r_k(z_1))^2  \label{eq:Nk4} \\
			&\qquad + \eta \Im m_n(z_1) (\sum_{i}^{(k)} a_{ik} -E - \Re m_n(z_2) -\Re r_k(z_2))^2 \nonumber\\
			&\qquad - \frac{\eta}{\sqrt{2}} \Im m_n(z_2) (\sum_{i}^{(k)} a_{ik} - E - \Re m_n(z_1) - \Re r_k(z_1))^2 . \label{eq:Nk5}
		\end{align}
	and 
	\begin{align*}
		B_{k} &= \left[(\sum_i^{(k)} a_{ik} - E - \Re m_n(z_1) - \Re r_k(z_1))^2 + (\eta + \Im m_n(z_1) + \Im r_k(z_1))^2\right] \\
		&\qquad \times \left[(\sum_i^{(k)} a_{ik} - E - \Re m_n(z_2) - \Re r_k(z_2))^2 + (\sqrt 2 \eta + \Im m_n(z_2) + \Im r_k(z_2))^2 \right].
	\end{align*}
	Our goal is to control $N_k$ from above. By Lemma \ref{lem:weak local law} the terms in \eqref{eq:Nk1} are $O(\eta^4)$.  By the same reasoning, \eqref{eq:Nk2} and \eqref{eq:Nk3} are $O(\eta^4)$.  We recall that on the event $\mathcal{E}$, by Lemma \cite[Lemma 7.1]{campbell2022extreme}\footnote{The statement of \cite[Theorem 7.1]{campbell2022extreme} is for a model of $A+D$ where the two matrices are independent and involves taking the expected value over $A$. However, the proof uses only a counting proposition similar to Proposition \ref{prop:eigen count} and the first step is to prove \eqref{eq:remdiff} before taking the expected value.},
	\begin{equation} \label{eq:remdiff}
	\sup_{z_1 = E + \ii \eta \in \tilde{S}_\delta} | \Re m_n(z_1) - \Re m_n(z_2) | = o(n^{-1/2}). 
	\end{equation}

	Thus, 
	\begin{align*}
		\eqref{eq:Nk4} &= \eta^2(\Re m_n(z_2) + \Re m_n(z_1)) (\Re m_n(z_2) - \Re m_n(z_1)) \\
		&\quad + \eta^2 (\Re r_k(z_2)^2 - \Re r_k(z_1)^2) \\
		&\quad + 2 \eta^2 (\sum_i^{(k)} a_{ik} - E)  (\Re m_n(z_1) - \Re m_n(z_2)) \\
		&\quad + 2 \eta^2 (\sum_i^{(k)} a_{ik} - E)  (\Re r_k(z_1) - \Re r_k(z_2)) \\
		&\quad + 2 \eta^2 (\Re m_n(z_2) \Re r_k(z_2) - \Re m_n(z_1) \Re r_k(z_1)) \\
		&\quad + \eta \Im r_k(z_1) (\sum_i^{(k)} a_{ik} - E - \Re m_n(z_2) - \Re r_k(z_2))^2 \\
		&\quad - \frac{\eta}{\sqrt 2} \Im r_k(z_2) (\sum_i^{(k)} a_{ik} - E - \Re m_n(z_1) - \Re r_k(z_1))^2 \\
		&\lesssim \eta^4  \log^{3/2} n,
	\end{align*}
	on $\mathcal{E}$, where the last line follows from \eqref{eq:remdiff}, Lemma \ref{cor:imr}, and Lemma \ref{lem:diff Im r}.	
	To handle $\eqref{eq:Nk5}$ we first observe that
	\begin{align}
	|\eta \Im m_n(z_1) - \frac{\eta}{\sqrt{2}} \Im m_n(z_2)| \oindicator{\mathcal{E}} &=  \oindicator{\mathcal{E}} \frac{1}{n}\sum_j \frac{\eta^2}{(\lambda_j -E)^2 + \eta^2} - \frac{\eta^2}{(\lambda_j - E)^2 + 2 \eta^2} \nonumber \\
	&= \frac{1}{n}\sum_j \frac{\eta^4  \oindicator{\mathcal{E}}}{((\lambda_j - E)^2 + \eta^2)((\lambda_j - E)^2 + 2 \eta^2)} \nonumber \\
	&		=\frac{1}{n}I_1\oindicator{\mathcal{E}}\label{eq:im m diff equality}		\\
	&\lesssim \eta^4n^{\eps} \label{eq:immdiff}.
	\end{align}
Therefore, 
	\begin{align*}
		\eqref{eq:Nk5} &= \frac{1}{n}I_1 (\sum_{i}^{(k)} a_{ik} - E)^2 \\
		&\qquad + \frac{\eta}{\sqrt{2}} \Im m_n(z_2) (\Re m_n(z_2) + \Re m_n(z_1)) (\Re m_n(z_2) - \Re m_n(z_1)) \\
		&\qquad + \eta (\Im m_n(z_1) - \frac{1}{\sqrt{2}} \Im m_n(z_2)) \Re m_n(z_2)^2 \\
		&\qquad + \eta (\Im m_n(z_1) - \frac{1}{\sqrt 2}\Im m_n(z_2)) \Re r_k(z_1)^2 \\
		&\qquad + \eta \Im m_n(z_1) (\Re r_k(z_2) + \Re r_k(z_1))(\Re r_k(z_2) - \Re r_k(z_1)) \\
		&\qquad + (\eta \Im m_n(z_1) - \frac{\eta}{\sqrt{2}} \Im m_n(z_2)) (\sum_i^{(k)} a_ik - E) \Re m_n(z_1) \\
		&\qquad +\eta \Im m_n(z_1) (\sum_i^{(k)} a_{ik} - E) (\Re m_n(z_2) - \Re m_n(z_1)) \\
		&\qquad + (\eta \Im m_n(z_1) - \frac{\eta}{\sqrt{2}} \Im m_n(z_2)) (\sum_i^{(k)} a_ik - E) \Re r_k(z_1) \\
		&\qquad + \eta \Im m_n(z_1) (\sum_i^{(k)} a_{ik} - E) (\Re r_k(z_2) - \Re r_k(z_1)) \\
		&\qquad + (\eta \Im m_n(z_1) - \frac{\eta}{\sqrt{2}} \Im m_n(z_2)) \Re m_n(z_1) \Re r_k(z_1) \\
		&\qquad + \eta \Im m_n(z_1) \Re m_n(z_1) (\Re r_k(z_2) - \Re r_k (z_1)) \\
		&\qquad + \eta \Im m_n(z_1) (\Re m_n(z_2) - \Re m_n(z_1)) \Re r_k(z_1) \\
		&=\frac{1}{n}I_{1}(\sum_{i}^{(k)} a_{ik} - E)^2 +O(\eta^4),
	\end{align*}
	on $\mathcal{E}$, where the last line follows from Lemma \ref{lem:weak local law},  \eqref{eq:remdiff}, \eqref{eq:immdiff}, and Lemma \ref{lem:r_k}.  We have now shown that 
	\[
	N_k\oindicator{\mathcal{E}} = \frac{1}{n}I_{1}(\sum_{i}^{(k)} a_{ik} - E)^2+O(\eta^4 \log^{3/2} n).
	\]
	Finally, on the event $\mathcal{E}$ and for $k \in J_1$, 
	\[
	|B_k| \gtrsim \log^2 n.  
	\]
	We can now conclude that 
	\[
\begin{aligned}
		\oindicator{\mathcal{E}}\eta \sum_{k\in J_1} \Im (G_{kk}(z_1) - \frac{G_{kk}(z_2)}{\sqrt 2} ) &:=\oindicator{\mathcal{E}} \sum_{k \in J_1} \frac{N_{k}}{B_{k}}\\
		&\lesssim \frac{n^{1}\eta^4 }{\log^{1/2}n} + I_n\frac{1}{n}\sum_{k\in J_1}\frac{(\sum_{i}^{(k)} a_{ik} - E)^2}{\log^2n}\\
		&=p_nI_1+ o(1),
\end{aligned}  
	\] where $p_n\rightarrow0$. 
	\end{proof}

	\begin{lemma}  There exists a sequence $q_n\rightarrow0$ such that 
 	\begin{align*}
 	\oindicator{\mathcal{E}}\Bigg|q_nI_2-\sum_{k \in J_1} &\frac{\eta^2+ \eta \Im  m_n(z_1)}{(\sum_i^{(k)} a_{ik} - E - \Re  m_n(z_1))^2+ (\eta + \Im  m_n(z_1))^2}  \\
 	&\qquad + \sum_{k \in J_1} \frac{\eta^2+ \frac{1}{\sqrt{2}}\eta \Im m_n(z_2)}{(\sum_i^{(k)} a_{ik} - E - \Re  m_n(z_2))^2 +(\sqrt{2} \eta + \Im  m_n(z_2))^2}	\Bigg|\\
 	&= o(1).
 	\end{align*}
 	\end{lemma}	
 	
 	\begin{proof}
	The proof of this lemma is similar to that of Lemma \ref{lem:imdifference} except even simpler as there are no $r_k$ terms.  Therefore, we omit the proof.  	
	\end{proof}
	
	We now control the terms from $k \in J_2$.  
	\begin{lemma}
		With overwhelming probability,
		\begin{align*} \label{eq:J2terms} 
			 &\Bigg|  \sum_{k \in J_2} \eta \Im(G_{kk}(z_1)) - \frac{\eta^2+ \eta \Im m_n(z_1)}{(\sum_i^{(k)} a_{ik} - E - \Re m_n(z_1))^2 +(\eta + \Im m_n(z_1))^2} \\
			 &\qquad - \sum_{k \in J_2} \frac{\eta}{2} \Im(G_{kk}(z_2)) + \frac{\eta^2+ \frac{1}{\sqrt{2}}\eta \Im m_n(z_2)}{(\sum_i^{(k)} a_{ik} - E - \Re m_n(z_2))^2 +(\sqrt{2} \eta + \Im m_n(z_2))^2} \Bigg| = o(1)	
		\end{align*}
	\end{lemma} 	
 	\begin{proof}
 		We establish \eqref{eq:J2terms} by recalling \eqref{eq:schurexpansion}
 		\begin{align} 
 			G_{kk}(z)  &=  \frac{1}{ \sum_{i}^{(k)} a_{ik} - z- m_n(z) - r_k} 
 		\end{align}
 	and that $\oindicator{\mathcal{E}}|r_k| = O(n^{-1/2} \log^{2\xi} n)$ by Lemma \ref{lem:r_k}.
 	Thus, 
 	\begin{align*}
 		&\sum_{k \in J_2} \eta \Im G_{kk}(z_1) - \frac{\eta^2+ \eta \Im m_n(z_1)}{(\sum_i^{(k)} a_{ik} - E - \Re m_n(z_1))^2 +(\eta + \Im m_n(z_1))^2} \\
 		&=  \frac{\eta^2 + \eta \Im(m_n(z_1) + \eta \Im r_k(z_1)}{(\sum_i^{(k)} a_{ik} - E - \Re m_n(z_1) - \Re r_k(z_1))^2 + (\eta + \Im m_n(z_1) + \Im r_k(z_1))^2} \\
 		&\qquad - \frac{\eta^2+ \eta \Im m_n(z_1)}{(\sum_i^{(k)} a_{ik} - E - \Re m_n(z_1))^2 +(\eta + \Im m_n(z_1))^2} \\
 		&:= \sum_{j \in J_2} \frac{N_k'}{B_k'}
 	\end{align*}
 	where on the event $\mathcal{E}$,
 	\begin{align*}
 		N_k' &= -(\eta^2 + \eta \Im m_n(z_1)) (\sum_{i}^{(k)} a_{ik} - E - \Re m_n(z_1))^2 \Re r_k(z_1) \\
 		&\qquad - (\eta^2 + \eta \Im m_n(z_1)) (\Re r_k(z_1))^2 \\
 		&\qquad + \eta \Im r_k(z_1) \left((\sum_i^{(k)} a_{ik} - E - \Re m_n(z_1))^2 + (\eta + \Im m_n(z_1))^2 \right) \\
 		&\qquad + (\eta^2 + \eta \Im m_n(z_1)) (\eta + \Im m_n(z_1)) \Im r_k(z_1) \\
 		&\qquad + (\eta^2 + \eta \Im m_n(z_1)) (\Im r_k(z_1))^2 \\
 		&\lesssim \eta^4n^{\eps/2}\log^{2\xi}n  (\sum_i^{(k)} a_{ik} - E - \Re m_n(z_1))^2 +  o(\eta^5)
 	\end{align*}
 	Additionally, on the event $\mathcal{E}$,
 	\begin{align*}
 		B_k' &:= \left( (\sum_i^{(k)} a_{ik} - E - \Re m_n(z_1) - \Re r_k(z_1))^2 + (\eta + \Im m_n(z_1) + \Im r_k(z_1))^2 \right) \\
 		&\qquad \times \left((\sum_i^{(k)} a_{ik} - E - \Re m_n(z_1))^2 +(\eta + \Im m_n(z_1))^2 \right) \\
 		&\gtrsim \eta^2(\sum_i^{(k)} a_{ik} - E - \Re m_n(z_1))^2+\eta^4.
 	\end{align*}
 	Therefore,
 	\[
 	\oindicator{\mathcal{E}}\left| \frac{N'_k}{B'_k} \right| \lesssim \eta^2n^{\eps/2}\log^{2\xi}n = O(\eta).
 	\]
 	We can now finally conclude that 
 	\[
 	\oindicator{\mathcal{E}}|\sum_{j \in J_2} \frac{N_k'}{B_k'}| \lesssim n^{\varepsilon} \eta = o(1).
 	\]
 	
 	Analogously,
 	\[
 	\sum_{k \in J_2} \frac{\eta}{2} \Im(G_{kk}(z_2)) - \frac{\eta^2+ \frac{1}{\sqrt{2}}\eta \Im m_n(z_2)}{(\sum_i^{(k)} a_{ik} - E - \Re m_n(z_2))^2 +(\sqrt{2} \eta + \Im m_n(z_2))^2} = o(1),
 	\]
 	on $\mathcal{E}$. This completes the proof.
 	\end{proof}
 	
 	Combining the above lemmas with the triangle inequality proves that with overwhelming probability
 	\begin{equation}
 		\sup_{z_1 = E + \ii \eta \in \tilde{S}_\delta}|(p_n+q_n+1)I_1 - I_2| = o(1).
 	\end{equation}
 	 We next write \begin{equation}\label{eq:A:I2 decomp}
 		I_2=\sum_{k=1}^n\frac{N_{k}^{''}}{B_{k}^{''}}
 	\end{equation} where \begin{equation}\label{eq:A:N'' def}
 	\begin{aligned}
 		N_{k}^{''}&:=\eta^2\left[\left(\sqrt{2}\eta+\Im m_n(z_2) \right)^2-\left(\eta+\Im m_n(z_1) \right)^2 \right]\\
 		&\quad+\eta\Im m_n(z_1)\left(\sqrt{2}\eta+\Im m_n(z_2)\right)^2 \\
 		&\quad -\frac{1}{\sqrt{2}}\eta\Im m_n(z_2)\left(\eta+\Im m_n(z_1)\right)^2\\
 		&\quad+\eta^2\left[\left(\sum_{i}^{(k)}a_{ik}-E-\Re m_n(z_2) \right)^2-\left(\sum_{i}^{(k)}a_{ik}-E-\Re m_n(z_1) \right)^2 \right]\\
 		&\quad+\eta\Im m_n(z_1)\left(\sum_{i}^{(k)}a_{ik}-E-\Re m_n(z_2)\right)^2 \\
 		&\quad-\frac{1}{\sqrt{2}}\eta\Im m_n(z_2)\left(\sum_{i}^{(k)}a_{ik}-E-\Re m_n(z_1)\right)^2. 
 	\end{aligned}
 \end{equation} and \begin{equation}
 \begin{aligned}
 	B_{k}^{''}&:=\left(\left(\sum_{i}^{(k)}a_{ik}-E-\Re m_n(z_1)\right)^2+\left(\eta+\Im  m_n(z_1)\right)^2\right) \\
 	&\qquad \times \left(\left(\sum_{i}^{(k)}a_{ik}-E-\Re m_n(z_2)\right)^2+\left(\sqrt{2}\eta+\Im m_n(z_2)\right)^2\right).
 \end{aligned}
\end{equation} What follows next is similar to the above and the argument of \cite[Section 7]{campbell2022extreme}. We begin by expressing $N_{k}''$ in a more manageable way. From Lemma \ref{lem:weak local law}, the first three lines of \eqref{eq:A:N'' def} are $\eta^{4}(1+o(1))$ uniformly in $k$ on $\mathcal{E}$. From \eqref{eq:remdiff} \begin{equation*}
\begin{aligned}
	&\left(\sum_{i}^{(k)}a_{ik}-E-\Re m_n(z_2) \right)^2-\left(\sum_{i}^{(k)}a_{ik}-E-\Re m_n(z_1) \right)^2\\&\quad=o\left(\eta^2\left(\sum_{i}^{(k)}a_{ik}-E-\Re m_n(z_1) \right)\right),
\end{aligned}
\end{equation*}  on $\mathcal{E}$.  We can apply \eqref{eq:immdiff} and \eqref{eq:remdiff} to conclude that the final difference in $N_{k}''$ is \begin{equation*}
\frac{1}{n}I_1\left(\sum_{i}^{(k)}a_{ik}-E\right)^2+o(\eta^4).
\end{equation*} Thus, \begin{equation*}
\begin{aligned}
	I_2&=\frac{1}{n}I_1\sum_{k=1}^{n}\frac{\left(\sum_{i}^{(k)}a_{ik}-E\right)^2}{B_{k}''}+(1+o(1))\sum_{k=1}^{n}\frac{\eta^4}{B_{k}''}\\
	&=p_n'I_1+(1+o(1))\sum_{k=1}^{n}\frac{\eta^4}{B_{k}''}
\end{aligned}
\end{equation*} for some $p_n'\rightarrow0$. Thus, \begin{equation}\label{eq:A:I1 eig detector}
\sup_{z_1 = E + \ii \eta \in \tilde{S}_\delta}\left|(1+o(1))I_{1}-(1+o(1))\sum_{k=1}^{n}\frac{\eta^4}{B_{k}^{''}} \right|=o(1),
\end{equation} on $\mathcal{E}$. The utility of \eqref{eq:A:I1 eig detector} is that $I_1$ is bounded away from zero for $E$ near an eigenvalue, while the second term $\sum_{k=1}^{n}\frac{\eta^4}{B_{k}^{''}}$ is bounded away from zero for $E$ near values determined by the large diagonal entries of $D$. 

We next need some basic asymptotic estimates for $m_n$. We make another adjustment to $\mathcal{E}$ by taking the intersection with the event \begin{equation*}
	\left\{\left|\left\{i\in[n]: L_{ii}\geq (2\log n)^{\frac{1}{2}-\eps''} \right\} \right|\leq\frac{n}{\sqrt{\log n}} \right\},
\end{equation*} for some fixed but arbitrary $\eps''\in (0,1/2)$, which holds with overwhelming probability by Proposition \ref{prop:eigen count}. 
\begin{lemma}\label{lem:expansion for m_n}
	\begin{equation}\label{eq:m expansion}
		\oindicator{\mathcal{E}}\sup_{z = E + \ii \eta \in \tilde{S}_\delta}|z|\left|m_{n}(z)+\frac{1}{z} \right|=o(1),
	\end{equation} and \begin{equation}\label{eq:expansion for m der}
	\oindicator{\mathcal{E}}\sup_{z = E + \ii \eta \in \tilde{S}_\delta}|z|^{2}\left|m_{n}'(z)-\frac{1}{z^2} \right|=o(1).
\end{equation}
\end{lemma}

\begin{proof}
	First we move $\frac{1}{z}$ inside the summation defining $m_n$ to say that \begin{equation}\label{eq:m_n expansion algebra}
		m_n(z)+\frac{1}{z}=\frac{1}{n}\sum_{k=1}^{n}\frac{1}{\lambda_{k}(L)-z}+\frac{1}{z}=\frac{1}{n}\sum_{k=1}^{n}\frac{\lambda_{k}(L)}{z(\lambda_{k}(L)-z)}.
	\end{equation}
	
	Let 
\[
J_1 = \{i: |\lambda_i(L)| < (\log n)^{1/2-\eps''} \},
\]\[
J_2 = \{i\notin J_{1}: |\lambda_i(L)| < \sqrt{(2 - 2\delta) \log n} \},
\]
and
\[
J_3 = [n] \setminus (J_1\cup J_2).
\]
On $\mathcal{E}$, \begin{equation*}
	|J_{2}|=o(n),
\end{equation*} and 
\[
|J_3| \leq n^{\varepsilon}.
\]  Thus, we can break the summation in \eqref{eq:m_n expansion algebra} to conclude that 
\begin{equation}\label{eq:J1 m_n asym}
	\begin{aligned}
		\sup_{z = E + \ii \eta \in \tilde{S}_\delta}\left|\frac{1}{n}\sum_{k \in J_1}\frac{\lambda_{k}(L)}{z(\lambda_{k}(L)-z)}\right|
		&\lesssim \frac{1}{n}\sum_{k \in J_1}\left(\log n\right)^{-\frac{1}{2}-\eps'}\\
		&=o\left(\frac{1}{|z|}\right),
	\end{aligned} 
\end{equation}  and  \begin{equation}\label{eq:J2 m_n asym}
\begin{aligned}
	\sup_{z = E + \ii \eta \in \tilde{S}_\delta}\left|\frac{1}{n}\sum_{k \in J_2}\frac{\lambda_{k}(L)}{z(\lambda_{k}(L)-z)}\right|&\lesssim\frac{1}{n}\sum_{k\in J_2}\frac{1}{\sqrt{\log n}}\\
	&=o\left(\frac{1}{\sqrt{\log n}}\right).
\end{aligned}
\end{equation} Finally, on $J_3$ \begin{equation}\label{eq:J3 m_n asym}
\begin{aligned}
	\sup_{z = E + \ii \eta \in \tilde{S}_\delta}\left|\frac{1}{n}\sum_{k \in J_3}\frac{\lambda_{k}(L)}{z(\lambda_{k}(L)-z)}\right|&\lesssim\frac{1}{n}\sum_{k\in J_3}\frac{1}{\eta}\\
	&=O\left(\frac{1}{n^{1-\eps}\eta} \right)\\
	&=o\left(\frac{1}{|z|}\right).
\end{aligned}
\end{equation}  Combining \eqref{eq:J1 m_n asym}, \eqref{eq:J2 m_n asym}, and \eqref{eq:J3 m_n asym} completes the proof of \eqref{eq:m expansion}. The proof of \eqref{eq:expansion for m der} follows in a nearly identical way.
\end{proof}

  For any $\tilde{D}_{kk}$  such that $\tilde{D}_{kk}+\ii n^{-1/4}\in\tilde{S}_{\delta}$, we let ${E}_{k}$ be such that \begin{equation}\label{eq:maser eq}
\tilde{D}_{kk}-{E}_{k}-\Re m_{n}\left({E}_{k}+\ii n^{-1/4}\right)=0.
\end{equation} Then, using Lemma \ref{lem:expansion for m_n}, 
\begin{equation}\label{eq:E locations}
{E}_{k}=\tilde{D}_{kk}+\frac{1}{\tilde{D}_{kk} }(1+o(1)).
\end{equation} Moreover, from \eqref{eq:remdiff} \begin{equation}
\tilde{D}_{kk}-{E}_{k}-\Re m_{n}\left({E}_{k}+\ii\sqrt{2}n^{-1/4}\right)=o\left(n^{-1/2}\right).
\end{equation} It then follows for any such $E_{k}$ that \begin{equation}\label{eq:Mfc lower bound}
\sup_{E\in\left[E_{k}-\frac{1}{a_n^{3/2}},E_{k}+\frac{1}{a_n^{3/2}} \right]} (1+o(1))\sum_{j=1}^{n}\frac{\eta^{4}}{B_{j}^{''}}\geq \frac{1}{2}.
\end{equation} We now consider the order statistics of the diagonal $D_{(1)}\geq D_{(2)}\geq\cdots\geq D_{(n)}$, and where defined we use the matching notation $E_{(1)}\geq E_{(2)}\geq\cdots\geq E_{(n)}$. By Proposition \ref{prop:eigen count} we can conclude that $D_{(1)}+\ii n^{-1/4}\in\tilde{S}_{\delta}$ with overwhelming probability, hence $E_{(1)}$ is well-defined. From \eqref{eq:E locations} and Corollary \ref{cor:diag order} we have that \begin{equation}
\begin{aligned}
	a_n\left(E_{(1)}-b_n \right)&=a_n\left(E_{(1)}-b_n'-\frac{1}{a_n} \right)\\
	&=a_n\left(D_{(1)}+\frac{1}{D_{(1)}}(1+o(1))-b_n'-\frac{1}{a_n} \right)\\
	&=a_n\left(D_{(1)}-b_n'+o(a_n^{-1}) \right),
\end{aligned}
\end{equation} asymptotically almost surely, where the fact that $\frac{D_{(1)}}{a_n}\rightarrow1$ in probability follows easily from the $k=1$ case of Corollary \ref{cor:diag order}.  

Additionally,  from \eqref{eq:Mfc lower bound}  \begin{equation}\label{eq:diagonal location eq}
\sup_{E\in\left[E_{(1)}-\frac{1}{a_n^{3/2}},E_{(1)}+\frac{1}{a_n^{3/2}} \right]} (1+o(1))\sum_{j=1}^{n}\frac{\eta^{4}}{B_{j}^{''}}\geq \frac{1}{2}.
\end{equation} From \eqref{eq:A:I1 eig detector} it also follows that  \begin{equation}\label{eq:I_1 first lower bound}
\sup_{E\in\left[E_{(1)}-\frac{1}{a_n^{3/2}},E_{(1)}+\frac{1}{a_n^{3/2}} \right]} (1+o(1))I_{1}\geq \frac{1}{2},
\end{equation} on the intersection of $\mathcal{E}$ and the event that \eqref{eq:diagonal location eq} holds. We update $\mathcal{E}$ again by intersecting with the event that \eqref{eq:diagonal location eq} holds. Assume, for the sake of contradiction, that on $\mathcal{E}$ there is not at least one eigenvalue of $L$ within $1/a_n^{3/2}$ of $E_{(1)}$.  If $J=\{j: \lambda_j(L)\geq\sqrt{(2-2\delta')\log n} \}$, then on $\mathcal{E}$ \begin{equation}\label{eq:A:eigenvalue too far}
I_1\leq o(1)\sum_{j\notin J}\eta^4+\sum_{j\in J}\frac{\eta^4}{1/a_n^6}=o(1)+O\left(\frac{a_n^6}{n^{99/100}} \right)=o(1),
\end{equation} a contradiction of \eqref{eq:I_1 first lower bound}. Hence, on the event $\mathcal{E}$ there is at least one eigenvalue in $\left[E_{(1)}-\frac{1}{a_n^{3/2}},E_{(1)}+\frac{1}{a_n^{3/2}} \right]$. A similar comparison proves that there cannot be an eigenvalue of $L$ beyond $E_{(1)}+a_{n}^{-3/2}$ on $\mathcal{E}$. 

We have shown that $|\lambda_{1}-E_{(1)}|=O(a_n^{-3/2})$ on $\mathcal{E}$. This estimate can be improved substantially. We will work exclusively on $\mathcal{E}$ for the remainder of the proof, and will no longer explicitly state this is the case. From Proposition \ref{prop:convergence of diagonal process}, $E_{(2)}\notin\left[E_{(1)}-\frac{1}{a_n^{3/2}},E_{(1)}+\frac{1}{a_n^{3/2}} \right]$ with probability $1-o(1)$, and hence \eqref{eq:Mfc lower bound} can be improved to conclude that \begin{equation}\label{eq:mfc better approx}
	\sup_{E\in\left[E_{(1)}-\frac{1}{a_n^{3/2}},E_{(1)}+\frac{1}{a_n^{3/2}} \right]} (1+o(1))\sum_{j=1}^{n}\frac{\eta^{4}}{B_{j}^{''}}= \frac{1}{2}+o(1),
\end{equation} asymptotically almost surely. It then follows from \eqref{eq:A:I1 eig detector} that \begin{equation}\label{eq:I_1 asym}
\sup_{E\in\left[E_{(1)}-\frac{1}{a_n^{3/2}},E_{(1)}+\frac{1}{a_n^{3/2}} \right]} (1+o(1))I_1= \frac{1}{2}+o(1),
\end{equation} asymptotically almost surely. Now let $c_n$ be such that $c_n\eta=E_{(1)}-\lambda_1(L)$. If $E=\lambda_{1}(L)$, it follows from Lemma \ref{lem:expansion for m_n} that \begin{equation}
(1+o(1))\sum_{j=1}^{n}\frac{\eta^4}{B_{j,n}}=\frac{1}{(1+c_n^2)(2+c_{n}^2)}+o(1),
\end{equation} asymptotically almost surely, a contradiction of \eqref{eq:mfc better approx} unless $c_n=o(1)$. In fact, for any eigenvalue $\lambda_j(L)\in\left[E_{(1)}-\frac{1}{a_n^{3/2}},E_{(1)}+\frac{1}{a_n^{3/2}} \right]$, the preceding argument shows $|\lambda_j(L)-E_{(1)}|=o(\eta)$. If there are $N$ eigenvalues of $L$ in the interval $\left[E_{(1)}-\frac{1}{a_n^{3/2}},E_{(1)}+\frac{1}{a_n^{3/2}} \right]$ and hence within $o(\eta)$ of $E_{(1)}$, then \begin{equation*}
\sup_{E\in\left[E_{(1)}-\frac{1}{a_n^{3/2}},E_{(1)}+\frac{1}{a_n^{3/2}} \right] }	\left|I_1\right|\geq\frac{N}{2}+o(1).
\end{equation*} We then see from \eqref{eq:A:I1 eig detector} and \eqref{eq:Mfc lower bound}  that $N=1$. 

For any fixed $k\in\N$, we can iterate the above argument $k$ times to conclude that \begin{equation}\label{eq:A:k eigenvalues comparison}
	\P\left( \|(\lambda_1,\lambda_2,\dots,\lambda_{k})-(E_{(1)},E_{(2)},\dots,E_{(k)})\|\geq\eta\right)=o(1).
\end{equation} Applying Proposition \ref{prop:convergence of diagonal process} to $\tilde{D}_{(1)},\tilde{D}_{(2)},\dots,\tilde{D}_{(k)}$ and \eqref{eq:E locations}  completes the proof of Theorem \ref{thm:largest k eigenvalues}.

	\bibliography{universality}

@article{MR2375601,
	author = {Houdr\'e, Christian and Xu, Hua},
	doi = {10.1214/EJP.v13-482},
	fjournal = {Electronic Journal of Probability},
	issn = {1083-6489},
	journal = {Electron. J. Probab.},
	mrclass = {60E07 (15A42 15A52 60F10)},
	mrnumber = {2375601},
	mrreviewer = {Ofer\ Zeitouni},
	pages = {no. 5, 107--134},
	title = {Concentration of the spectral measure for large random matrices with stable entries},
	url = {https://doi.org/10.1214/EJP.v13-482},
	volume = {13},
	year = {2008}}

@article{MR2373441,
	author = {Ben Arous, G\'erard and Guionnet, Alice},
	doi = {10.1007/s00220-007-0389-x},
	fjournal = {Communications in Mathematical Physics},
	issn = {0010-3616,1432-0916},
	journal = {Comm. Math. Phys.},
	mrclass = {60B10 (60F05)},
	mrnumber = {2373441},
	mrreviewer = {Michael\ Voit},
	number = {3},
	pages = {715--751},
	title = {The spectrum of heavy tailed random matrices},
	url = {https://doi.org/10.1007/s00220-007-0389-x},
	volume = {278},
	year = {2008}}

@article{MR3541852,
	author = {Bourgade, Paul and Erd\H{o}s, Laszlo and Yau, Horng-Tzer and Yin, Jun},
	doi = {10.1002/cpa.21624},
	fjournal = {Communications on Pure and Applied Mathematics},
	issn = {0010-3640,1097-0312},
	journal = {Comm. Pure Appl. Math.},
	mrclass = {60B20 (81R12)},
	mrnumber = {3541852},
	mrreviewer = {Nizar\ Demni},
	number = {10},
	pages = {1815--1881},
	title = {Fixed energy universality for generalized {W}igner matrices},
	url = {https://doi.org/10.1002/cpa.21624},
	volume = {69},
	year = {2016}}

@article{MR3478311,
	author = {Ajanki, Oskari H. and Erd\H{o}s, L\'aszl\'o{} and Kr\"uger, Torben},
	doi = {10.1007/s10955-016-1479-y},
	fjournal = {Journal of Statistical Physics},
	issn = {0022-4715,1572-9613},
	journal = {J. Stat. Phys.},
	mrclass = {60B20 (15B52)},
	mrnumber = {3478311},
	mrreviewer = {Olivier\ Marchal},
	number = {2},
	pages = {280--302},
	title = {Local spectral statistics of {G}aussian matrices with correlated entries},
	url = {https://doi.org/10.1007/s10955-016-1479-y},
	volume = {163},
	year = {2016}}

@article{MR3372074,
	author = {Erd\H{o}s, L\'aszl\'o{} and Yau, Horng-Tzer},
	doi = {10.4171/JEMS/548},
	fjournal = {Journal of the European Mathematical Society (JEMS)},
	issn = {1435-9855,1435-9863},
	journal = {J. Eur. Math. Soc. (JEMS)},
	mrclass = {82B44 (15B52)},
	mrnumber = {3372074},
	mrreviewer = {Sasha\ Sodin},
	number = {8},
	pages = {1927--2036},
	title = {Gap universality of generalized {W}igner and {$\beta$}-ensembles},
	url = {https://doi.org/10.4171/JEMS/548},
	volume = {17},
	year = {2015}}

@article{MR2981427,
	author = {Erd\H{o}s, L\'aszl\'o{} and Yau, Horng-Tzer and Yin, Jun},
	doi = {10.1007/s00440-011-0390-3},
	fjournal = {Probability Theory and Related Fields},
	issn = {0178-8051,1432-2064},
	journal = {Probab. Theory Related Fields},
	mrclass = {60B20 (82B44)},
	mrnumber = {2981427},
	mrreviewer = {Florent\ Benaych-Georges},
	number = {1-2},
	pages = {341--407},
	title = {Bulk universality for generalized {W}igner matrices},
	url = {https://doi.org/10.1007/s00440-011-0390-3},
	volume = {154},
	year = {2012}}

@incollection{MR2730790,
	author = {Erd\H{o}s, Laszl\'o},
	booktitle = {X{VI}th {I}nternational {C}ongress on {M}athematical {P}hysics},
	doi = {10.1142/9789814304634\_0004},
	isbn = {978-981-4304-62-7; 981-4304-62-X},
	mrclass = {82B27 (60B20 82B44)},
	mrnumber = {2730790},
	pages = {86--105},
	publisher = {World Sci. Publ., Hackensack, NJ},
	title = {Universality of {W}igner random matrices},
	url = {https://doi.org/10.1142/9789814304634_0004},
	year = {2010}}

@book{MR691492,
	author = {Leadbetter, M. R. and Lindgren, Georg and Rootz\'{e}n, Holger},
	isbn = {0-387-90731-9},
	mrclass = {60F05 (60G15 62G30)},
	mrnumber = {691492},
	mrreviewer = {Simeon M. Berman},
	pages = {xii+336},
	publisher = {Springer-Verlag, New York-Berlin},
	series = {Springer Series in Statistics},
	title = {Extremes and related properties of random sequences and processes},
	year = {1983}}

@book{MR900810,
	author = {Resnick, Sidney I.},
	doi = {10.1007/978-0-387-75953-1},
	isbn = {0-387-96481-9},
	mrclass = {60K99 (60G55)},
	mrnumber = {900810},
	mrreviewer = {Charles M. Goldie},
	pages = {xii+320},
	publisher = {Springer-Verlag, New York},
	series = {Applied Probability. A Series of the Applied Probability Trust},
	title = {Extreme values, regular variation, and point processes},
	url = {https://doi.org/10.1007/978-0-387-75953-1},
	volume = {4},
	year = {1987}}

@article{MR1257246,
	author = {Tracy, Craig A. and Widom, Harold},
	fjournal = {Communications in Mathematical Physics},
	issn = {0010-3616},
	journal = {Comm. Math. Phys.},
	mrclass = {82B05 (33C90 47A75 47G10 47N55 82B10)},
	mrnumber = {1257246},
	mrreviewer = {Estelle L. Basor},
	number = {1},
	pages = {151--174},
	title = {Level-spacing distributions and the {A}iry kernel},
	url = {http://projecteuclid.org/euclid.cmp/1104254495},
	volume = {159},
	year = {1994}}

@article{MR1385083,
	author = {Tracy, Craig A. and Widom, Harold},
	fjournal = {Communications in Mathematical Physics},
	issn = {0010-3616},
	journal = {Comm. Math. Phys.},
	mrclass = {82B44 (15A52 47N55 60H25)},
	mrnumber = {1385083},
	mrreviewer = {Oleksiy Khorunzhiy},
	number = {3},
	pages = {727--754},
	title = {On orthogonal and symplectic matrix ensembles},
	url = {http://projecteuclid.org/euclid.cmp/1104286442},
	volume = {177},
	year = {1996}}

@article{MR1727234,
	author = {Soshnikov, Alexander},
	doi = {10.1007/s002200050743},
	fjournal = {Communications in Mathematical Physics},
	issn = {0010-3616},
	journal = {Comm. Math. Phys.},
	mrclass = {82B41 (15A52 60F99 82B44)},
	mrnumber = {1727234},
	mrreviewer = {Boris A. Khoruzhenko},
	number = {3},
	pages = {697--733},
	title = {Universality at the edge of the spectrum in {W}igner random matrices},
	url = {https://doi.org/10.1007/s002200050743},
	volume = {207},
	year = {1999}}

@article{MR1747178,
	author = {Fyodorov, Yan V.},
	doi = {10.1088/0305-4470/32/42/314},
	fjournal = {Journal of Physics. A. Mathematical and General},
	issn = {0305-4470,1751-8121},
	journal = {J. Phys. A},
	mrclass = {82B44 (94C05)},
	mrnumber = {1747178},
	number = {42},
	pages = {7429--7446},
	title = {Spectral properties of random reactance networks and random matrix pencils},
	url = {https://doi.org/10.1088/0305-4470/32/42/314},
	volume = {32},
	year = {1999}}

@article{MR2081462,
	author = {Soshnikov, Alexander},
	doi = {10.1214/ECP.v9-1112},
	fjournal = {Electronic Communications in Probability},
	issn = {1083-589X},
	journal = {Electron. Comm. Probab.},
	mrclass = {60E07},
	mrnumber = {2081462},
	mrreviewer = {Mohamed\ Sifi},
	pages = {82--91},
	title = {Poisson statistics for the largest eigenvalues of {W}igner random matrices with heavy tails},
	url = {https://doi.org/10.1214/ECP.v9-1112},
	volume = {9},
	year = {2004}}

@article{MR2125574,
	author = {Soshnikov, Alexander and Fyodorov, Yan V.},
	doi = {10.1063/1.1855932},
	fjournal = {Journal of Mathematical Physics},
	issn = {0022-2488,1089-7658},
	journal = {J. Math. Phys.},
	mrclass = {82B41 (60G55 62E10)},
	mrnumber = {2125574},
	mrreviewer = {Razvan\ Teodorescu},
	number = {3},
	pages = {033302, 15},
	title = {On the largest singular values of random matrices with independent {C}auchy entries},
	url = {https://doi.org/10.1063/1.1855932},
	volume = {46},
	year = {2005}}

@article{MR2206341,
	author = {Bryc, W\l{o}dzimierz and Dembo, Amir and Jiang, Tiefeng},
	doi = {10.1214/009117905000000495},
	fjournal = {The Annals of Probability},
	issn = {0091-1798},
	journal = {Ann. Probab.},
	mrclass = {60F99 (15A52 60F10 62H10)},
	mrnumber = {2206341},
	mrreviewer = {Martin V. Hildebrand},
	number = {1},
	pages = {1--38},
	title = {Spectral measure of large random {H}ankel, {M}arkov and {T}oeplitz matrices},
	url = {https://doi.org/10.1214/009117905000000495},
	volume = {34},
	year = {2006}}

@incollection{MR2234922,
	author = {Soshnikov, Alexander},
	booktitle = {Mathematical physics of quantum mechanics},
	doi = {10.1007/3-540-34273-7\_26},
	isbn = {978-3-540-31026-6; 3-540-31026-6},
	mrclass = {82B41 (15A52 60F05 82B44)},
	mrnumber = {2234922},
	mrreviewer = {Razvan\ Teodorescu},
	pages = {351--364},
	publisher = {Springer, Berlin},
	series = {Lecture Notes in Phys.},
	title = {Poisson statistics for the largest eigenvalues in random matrix ensembles},
	url = {https://doi.org/10.1007/3-540-34273-7_26},
	volume = {690},
	year = {2006}}

@book{MR2248695,
	author = {Chung, Fan and Lu, Linyuan},
	doi = {10.1090/cbms/107},
	isbn = {978-0-8218-3657-6; 0-8218-3657-9},
	mrclass = {05C85 (68R10 94C15)},
	mrnumber = {2248695},
	mrreviewer = {Colin D. Cooper},
	pages = {viii+264},
	publisher = {Published for the Conference Board of the Mathematical Sciences, Washington, DC; by the American Mathematical Society, Providence, RI},
	series = {CBMS Regional Conference Series in Mathematics},
	title = {Complex graphs and networks},
	url = {https://doi.org/10.1090/cbms/107},
	volume = {107},
	year = {2006}}

@article{MR2308592,
	author = {El Karoui, Noureddine},
	doi = {10.1214/009117906000000917},
	fjournal = {The Annals of Probability},
	issn = {0091-1798},
	journal = {Ann. Probab.},
	mrclass = {60F05 (47B35 62E20)},
	mrnumber = {2308592},
	mrreviewer = {Ofer Zeitouni},
	number = {2},
	pages = {663--714},
	title = {Tracy-{W}idom limit for the largest eigenvalue of a large class of complex sample covariance matrices},
	url = {https://doi.org/10.1214/009117906000000917},
	volume = {35},
	year = {2007}}

@article{MR2398763,
	author = {Onatski, Alexei},
	doi = {10.1214/07-AAP454},
	fjournal = {The Annals of Applied Probability},
	issn = {1050-5164},
	journal = {Ann. Appl. Probab.},
	mrclass = {60F05 (62E20)},
	mrnumber = {2398763},
	mrreviewer = {Alexander Roitershtein},
	number = {2},
	pages = {470--490},
	title = {The {T}racy-{W}idom limit for the largest eigenvalues of singular complex {W}ishart matrices},
	url = {https://doi.org/10.1214/07-AAP454},
	volume = {18},
	year = {2008}}

@article{MR2485010,
	author = {Johnstone, Iain M.},
	doi = {10.1214/08-AOS605},
	fjournal = {The Annals of Statistics},
	issn = {0090-5364},
	journal = {Ann. Statist.},
	mrclass = {62H10 (60B20 62E20)},
	mrnumber = {2485010},
	mrreviewer = {Carlos Martins-Filho},
	number = {6},
	pages = {2638--2716},
	title = {Multivariate analysis and {J}acobi ensembles: largest eigenvalue, {T}racy-{W}idom limits and rates of convergence},
	url = {https://doi.org/10.1214/08-AOS605},
	volume = {36},
	year = {2008}}

@article{MR2545551,
	author = {Sodin, Sasha},
	doi = {10.1007/s10955-009-9813-2},
	fjournal = {Journal of Statistical Physics},
	issn = {0022-4715},
	journal = {J. Stat. Phys.},
	mrclass = {60B20 (82B44)},
	mrnumber = {2545551},
	mrreviewer = {VenKata K. B. Kota},
	number = {5},
	pages = {834--841},
	title = {The {T}racy-{W}idom law for some sparse random matrices},
	url = {https://doi.org/10.1007/s10955-009-9813-2},
	volume = {136},
	year = {2009}}

@article{MR2548495,
	author = {Auffinger, Antonio and Ben Arous, G\'{e}rard and P\'{e}ch\'{e}, Sandrine},
	doi = {10.1214/08-AIHP188},
	fjournal = {Annales de l'Institut Henri Poincar\'{e} Probabilit\'{e}s et Statistiques},
	issn = {0246-0203,1778-7017},
	journal = {Ann. Inst. Henri Poincar\'{e} Probab. Stat.},
	mrclass = {60B20 (15B52 60G52)},
	mrnumber = {2548495},
	mrreviewer = {Walid\ Hachem},
	number = {3},
	pages = {589--610},
	title = {Poisson convergence for the largest eigenvalues of heavy tailed random matrices},
	url = {https://doi.org/10.1214/08-AIHP188},
	volume = {45},
	year = {2009}}

@book{MR2567175,
	author = {Bai, Zhidong and Silverstein, Jack W.},
	doi = {10.1007/978-1-4419-0661-8},
	edition = {Second},
	isbn = {978-1-4419-0660-1},
	mrclass = {60B20 (15B52 62H99 91G70 94A05)},
	mrnumber = {2567175},
	mrreviewer = {Wenbo V. Li},
	pages = {xvi+551},
	publisher = {Springer, New York},
	series = {Springer Series in Statistics},
	title = {Spectral analysis of large dimensional random matrices},
	url = {https://doi.org/10.1007/978-1-4419-0661-8},
	year = {2010}}

@article{MR2669449,
	author = {Tao, Terence and Vu, Van},
	doi = {10.1007/s00220-010-1044-5},
	fjournal = {Communications in Mathematical Physics},
	issn = {0010-3616},
	journal = {Comm. Math. Phys.},
	mrclass = {60B20},
	mrnumber = {2669449},
	mrreviewer = {Nizar Demni},
	number = {2},
	pages = {549--572},
	title = {Random matrices: universality of local eigenvalue statistics up to the edge},
	url = {https://doi.org/10.1007/s00220-010-1044-5},
	volume = {298},
	year = {2010}}

@article{MR2759729,
	author = {Ding, Xue and Jiang, Tiefeng},
	doi = {10.1214/10-AAP677},
	fjournal = {The Annals of Applied Probability},
	issn = {1050-5164},
	journal = {Ann. Appl. Probab.},
	mrclass = {60B20 (05C50 05C80 60B10)},
	mrnumber = {2759729},
	mrreviewer = {Steven Joel Miller},
	number = {6},
	pages = {2086--2117},
	title = {Spectral distributions of adjacency and {L}aplacian matrices of random graphs},
	url = {https://doi.org/10.1214/10-AAP677},
	volume = {20},
	year = {2010}}

@article{MR2784665,
	author = {Tao, Terence and Vu, Van},
	doi = {10.1007/s11511-011-0061-3},
	fjournal = {Acta Mathematica},
	issn = {0001-5962},
	journal = {Acta Math.},
	mrclass = {60B20},
	mrnumber = {2784665},
	mrreviewer = {Djalil Chafa\"{\i}},
	number = {1},
	pages = {127--204},
	title = {Random matrices: universality of local eigenvalue statistics},
	url = {https://doi.org/10.1007/s11511-011-0061-3},
	volume = {206},
	year = {2011}}

@book{MR2906465,
	author = {Tao, Terence},
	doi = {10.1090/gsm/132},
	isbn = {978-0-8218-7430-1},
	mrclass = {60B20 (15B52)},
	mrnumber = {2906465},
	mrreviewer = {Steven Joel Miller},
	pages = {x+282},
	publisher = {American Mathematical Society, Providence, RI},
	series = {Graduate Studies in Mathematics},
	title = {Topics in random matrix theory},
	url = {https://doi.org/10.1090/gsm/132},
	volume = {132},
	year = {2012}}

@article{MR2948689,
	author = {Jiang, Tiefeng},
	doi = {10.1007/s00440-011-0357-4},
	fjournal = {Probability Theory and Related Fields},
	issn = {0178-8051},
	journal = {Probab. Theory Related Fields},
	mrclass = {05C80 (05C50 60B20 60F15 62G30)},
	mrnumber = {2948689},
	mrreviewer = {Jonathan Cutler},
	number = {3-4},
	pages = {671--690},
	title = {Low eigenvalues of {L}aplacian matrices of large random graphs},
	url = {https://doi.org/10.1007/s00440-011-0357-4},
	volume = {153},
	year = {2012}}

@article{MR2967963,
	author = {Jiang, Tiefeng},
	doi = {10.1142/S2010326312500049},
	fjournal = {Random Matrices. Theory and Applications},
	issn = {2010-3263},
	journal = {Random Matrices Theory Appl.},
	mrclass = {05C80 (05C50 60B10 60B20)},
	mrnumber = {2967963},
	number = {3},
	pages = {1250004, 20},
	title = {Empirical distributions of {L}aplacian matrices of large dilute random graphs},
	url = {https://doi.org/10.1142/S2010326312500049},
	volume = {1},
	year = {2012}}

@article{MR2988403,
	author = {Bao, Zhigang and Pan, Guangming and Zhou, Wang},
	doi = {10.1214/EJP.v17-1962},
	fjournal = {Electronic Journal of Probability},
	journal = {Electron. J. Probab.},
	mrclass = {62H10 (15B52 60B20)},
	mrnumber = {2988403},
	mrreviewer = {Antonio Auffinger},
	pages = {no. 88, 32},
	title = {Tracy-{W}idom law for the extreme eigenvalues of sample correlation matrices},
	url = {https://doi.org/10.1214/EJP.v17-1962},
	volume = {17},
	year = {2012}}

@article{MR3068390,
	author = {Erd\H{o}s, L\'{a}szl\'{o} and Knowles, Antti and Yau, Horng-Tzer and Yin, Jun},
	doi = {10.1214/EJP.v18-2473},
	fjournal = {Electronic Journal of Probability},
	journal = {Electron. J. Probab.},
	mrclass = {60B20},
	mrnumber = {3068390},
	pages = {no. 59, 58},
	title = {The local semicircle law for a general class of random matrices},
	url = {https://doi.org/10.1214/EJP.v18-2473},
	volume = {18},
	year = {2013}}

@article{MR3161313,
	author = {Lee, Ji Oon and Yin, Jun},
	doi = {10.1215/00127094-2414767},
	fjournal = {Duke Mathematical Journal},
	issn = {0012-7094},
	journal = {Duke Math. J.},
	mrclass = {60B20},
	mrnumber = {3161313},
	mrreviewer = {Florent Benaych-Georges},
	number = {1},
	pages = {117--173},
	title = {A necessary and sufficient condition for edge universality of {W}igner matrices},
	url = {https://doi.org/10.1215/00127094-2414767},
	volume = {163},
	year = {2014}}

@article{MR3321497,
	author = {Ding, Xue},
	doi = {10.1016/j.spl.2015.01.005},
	fjournal = {Statistics \& Probability Letters},
	issn = {0167-7152},
	journal = {Statist. Probab. Lett.},
	mrclass = {15B52 (60B20 60F05 60F15)},
	mrnumber = {3321497},
	mrreviewer = {Bin Ye},
	pages = {61--69},
	title = {On some spectral properties of large block {L}aplacian random matrices},
	url = {https://doi.org/10.1016/j.spl.2015.01.005},
	volume = {99},
	year = {2015}}

@article{MR3349181,
	author = {Abbe, Emmanuel and Bandeira, Afonso S. and Bracher, Annina and Singer, Amit},
	doi = {10.1109/TNSE.2014.2368716},
	fjournal = {IEEE Transactions on Network Science and Engineering},
	journal = {IEEE Trans. Network Sci. Eng.},
	mrclass = {94C15},
	mrnumber = {3349181},
	number = {1},
	pages = {10--22},
	title = {Decoding binary node labels from censored edge measurements: phase transition and efficient recovery},
	url = {https://doi.org/10.1109/TNSE.2014.2368716},
	volume = {1},
	year = {2014}}

@article{MR3582818,
	author = {Lee, Ji Oon and Schnelli, Kevin},
	doi = {10.1214/16-AAP1193},
	fjournal = {The Annals of Applied Probability},
	issn = {1050-5164},
	journal = {Ann. Appl. Probab.},
	mrclass = {60B20 (15B52 62H10)},
	mrnumber = {3582818},
	mrreviewer = {Tatyana Shcherbina},
	number = {6},
	pages = {3786--3839},
	title = {Tracy-{W}idom distribution for the largest eigenvalue of real sample covariance matrices with general population},
	url = {https://doi.org/10.1214/16-AAP1193},
	volume = {26},
	year = {2016}}

@book{MR3699468,
	author = {Erd\H{o}s, L\'{a}szl\'{o} and Yau, Horng-Tzer},
	isbn = {978-1-4704-3648-3},
	mrclass = {60B20 (15B52 60F10 82B44)},
	mrnumber = {3699468},
	mrreviewer = {Sasha Sodin},
	pages = {ix+226},
	publisher = {Courant Institute of Mathematical Sciences, New York; American Mathematical Society, Providence, RI},
	series = {Courant Lecture Notes in Mathematics},
	title = {A dynamical approach to random matrix theory},
	volume = {28},
	year = {2017}}

@article{MR3777782,
	author = {Bandeira, Afonso S.},
	doi = {10.1007/s10208-016-9341-9},
	fjournal = {Foundations of Computational Mathematics. The Journal of the Society for the Foundations of Computational Mathematics},
	issn = {1615-3375},
	journal = {Found. Comput. Math.},
	mrclass = {15B52 (05C80 05C85 90C22)},
	mrnumber = {3777782},
	mrreviewer = {Sho Matsumoto},
	number = {2},
	pages = {345--379},
	title = {Random {L}aplacian matrices and convex relaxations},
	url = {https://doi.org/10.1007/s10208-016-9341-9},
	volume = {18},
	year = {2018}}

@article{MR3800840,
	author = {Lee, Ji Oon and Schnelli, Kevin},
	doi = {10.1007/s00440-017-0787-8},
	fjournal = {Probability Theory and Related Fields},
	issn = {0178-8051},
	journal = {Probab. Theory Related Fields},
	mrclass = {60B20 (62H10)},
	mrnumber = {3800840},
	number = {1-2},
	pages = {543--616},
	title = {Local law and {T}racy-{W}idom limit for sparse random matrices},
	url = {https://doi.org/10.1007/s00440-017-0787-8},
	volume = {171},
	year = {2018}}

@article{MR3916110,
	author = {Aggarwal, Amol},
	doi = {10.1007/s00440-018-0836-y},
	fjournal = {Probability Theory and Related Fields},
	issn = {0178-8051},
	journal = {Probab. Theory Related Fields},
	mrclass = {15B52 (60B20)},
	mrnumber = {3916110},
	mrreviewer = {Pavel T. Stoynov},
	number = {1-2},
	pages = {375--432},
	title = {Bulk universality for generalized {W}igner matrices with few moments},
	url = {https://doi.org/10.1007/s00440-018-0836-y},
	volume = {173},
	year = {2019}}

@article{MR4019881,
	author = {Hwang, Jong Yun and Lee, Ji Oon and Schnelli, Kevin},
	doi = {10.1214/19-AAP1472},
	fjournal = {The Annals of Applied Probability},
	issn = {1050-5164},
	journal = {Ann. Appl. Probab.},
	mrclass = {60B20 (62H10)},
	mrnumber = {4019881},
	mrreviewer = {Stepan Mazur},
	number = {5},
	pages = {3006--3036},
	title = {Local law and {T}racy-{W}idom limit for sparse sample covariance matrices},
	url = {https://doi.org/10.1214/19-AAP1472},
	volume = {29},
	year = {2019}}

@article{MR4058984,
	author = {Huang, Jiaoyang and Landon, Benjamin},
	doi = {10.1214/19-AIHP957},
	fjournal = {Annales de l'Institut Henri Poincar\'{e} Probabilit\'{e}s et Statistiques},
	issn = {0246-0203},
	journal = {Ann. Inst. Henri Poincar\'{e} Probab. Stat.},
	mrclass = {15B52 (05C50 60B20)},
	mrnumber = {4058984},
	mrreviewer = {Thomas Kriecherbauer},
	number = {1},
	pages = {120--154},
	title = {Spectral statistics of sparse {E}rd{\H{o}}s-{R}\'{e}nyi graph {L}aplacians},
	url = {https://doi.org/10.1214/19-AIHP957},
	volume = {56},
	year = {2020}}

@article{MR4091114,
	author = {Hwang, Jong Yun and Lee, Ji Oon and Yang, Wooseok},
	doi = {10.3150/20-BEJ1201},
	fjournal = {Bernoulli. Official Journal of the Bernoulli Society for Mathematical Statistics and Probability},
	issn = {1350-7265},
	journal = {Bernoulli},
	mrclass = {60B20 (05C80)},
	mrnumber = {4091114},
	mrreviewer = {Fernando Lucas Metz},
	number = {3},
	pages = {2400--2435},
	title = {Local law and {T}racy-{W}idom limit for sparse stochastic block models},
	url = {https://doi.org/10.3150/20-BEJ1201},
	volume = {26},
	year = {2020}}

@article{MR4193182,
	author = {Chakrabarty, Arijit and Hazra, Rajat Subhra and den Hollander, Frank and Sfragara, Matteo},
	doi = {10.1142/S201032632150009X},
	fjournal = {Random Matrices. Theory and Applications},
	issn = {2010-3263},
	journal = {Random Matrices Theory Appl.},
	mrclass = {60B20 (05C80 46L54)},
	mrnumber = {4193182},
	number = {1},
	pages = {Paper No. 2150009, 34},
	title = {Spectra of adjacency and {L}aplacian matrices of inhomogeneous {E}rd{\H{o}}s-{R}\'{e}nyi random graphs},
	url = {https://doi.org/10.1142/S201032632150009X},
	volume = {10},
	year = {2021}}

@article{MR4260217,
	author = {Tian, Boping and Zhang, Yangchun and Zhou, Wang},
	doi = {10.1142/S2010326321500222},
	fjournal = {Random Matrices. Theory and Applications},
	issn = {2010-3263},
	journal = {Random Matrices Theory Appl.},
	mrclass = {62M10 (60B20)},
	mrnumber = {4260217},
	number = {2},
	pages = {Paper No. 2150022, 22},
	title = {Tracy-{W}idom law for the largest eigenvalue of sample covariance matrix generated by {VARMA}},
	url = {https://doi.org/10.1142/S2010326321500222},
	volume = {10},
	year = {2021}}

@article{MR4440252,
	author = {Chatterjee, Anirban and Hazra, Rajat Subhra},
	doi = {10.1142/S2010326322500265},
	fjournal = {Random Matrices. Theory and Applications},
	issn = {2010-3263},
	journal = {Random Matrices Theory Appl.},
	mrclass = {60B20 (05C80 46L54)},
	mrnumber = {4440252},
	number = {3},
	pages = {Paper No. 2250026, 66},
	title = {Spectral properties for the {L}aplacian of a generalized {W}igner matrix},
	url = {https://doi.org/10.1142/S2010326322500265},
	volume = {11},
	year = {2022}}

@article{MR4474568,
	author = {Jun, Wen and Jiahui, Xie and Long, Yu and Wang, Zhou},
	doi = {10.3150/21-bej1443},
	fjournal = {Bernoulli. Official Journal of the Bernoulli Society for Mathematical Statistics and Probability},
	issn = {1350-7265},
	journal = {Bernoulli},
	mrclass = {60B20 (15A18 15B52 62H10)},
	mrnumber = {4474568},
	number = {4},
	pages = {2941--2967},
	title = {Tracy-{W}idom limit for the largest eigenvalue of high-dimensional covariance matrices in elliptical distributions},
	url = {https://doi.org/10.3150/21-bej1443},
	volume = {28},
	year = {2022}}

@article{PhysRevE.67.047101,
	author = {St\"aring, J. and Mehlig, B. and Fyodorov, Yan V. and Luck, J. M.},
	doi = {10.1103/PhysRevE.67.047101},
	issue = {4},
	journal = {Phys. Rev. E},
	month = {Apr},
	numpages = {4},
	pages = {047101},
	publisher = {American Physical Society},
	title = {Random symmetric matrices with a constraint: The spectral density of random impedance networks},
	url = {https://link.aps.org/doi/10.1103/PhysRevE.67.047101},
	volume = {67},
	year = {2003}}

@article{campbell2022extreme,
	author = {Campbell, Andrew and Luh, Kyle and O'Rourke, Sean},
	doi = {10.1214/25-ejp1366},
	fjournal = {Electronic Journal of Probability},
	issn = {1083-6489},
	journal = {Electron. J. Probab.},
	mrclass = {60B20 (05C50)},
	mrnumber = {4926015},
	mrreviewer = {Yan\ V.\ Fyodorov},
	note = {With an appendix by Santiago Arenas-Velilla and Victor Perez-Abreu},
	pages = {Paper No. 104, 52},
	title = {Extreme eigenvalues of {L}aplacian random matrices with {G}aussian entries},
	url = {https://doi.org/10.1214/25-ejp1366},
	volume = {30},
	year = {2025}}

@article{MR3210147,
	author = {Benaych-Georges, Florent and Guionnet, Alice and Male, Camille},
	doi = {10.1007/s00220-014-1975-3},
	fjournal = {Communications in Mathematical Physics},
	issn = {0010-3616,1432-0916},
	journal = {Comm. Math. Phys.},
	mrclass = {60B20 (15B52 60F05)},
	mrnumber = {3210147},
	mrreviewer = {Jiang\ Hu},
	number = {2},
	pages = {641--686},
	title = {Central limit theorems for linear statistics of heavy tailed random matrices},
	url = {https://doi.org/10.1007/s00220-014-1975-3},
	volume = {329},
	year = {2014}}

@book{Vershynin2018,
	author = {Vershynin, Roman},
	doi = {10.1017/9781108231596},
	isbn = {978-1-108-41519-4},
	mrclass = {60-01 (60B05 60B20 60E15 60Fxx 62H25)},
	mrnumber = {3837109},
	mrreviewer = {Sasha\ Sodin},
	note = {An introduction with applications in data science, With a foreword by Sara van de Geer},
	pages = {xiv+284},
	publisher = {Cambridge University Press, Cambridge},
	series = {Cambridge Series in Statistical and Probabilistic Mathematics},
	title = {High-dimensional probability},
	url = {https://doi.org/10.1017/9781108231596},
	volume = {47},
	year = {2018}}

@book{Bollobas2001,
	author = {Bollob\'as, B\'ela},
	doi = {10.1017/CBO9780511814068},
	edition = {Second},
	isbn = {0-521-80920-7; 0-521-79722-5},
	mrclass = {05C80 (60C05)},
	mrnumber = {1864966},
	pages = {xviii+498},
	publisher = {Cambridge University Press, Cambridge},
	series = {Cambridge Studies in Advanced Mathematics},
	title = {Random graphs},
	url = {https://doi.org/10.1017/CBO9780511814068},
	volume = {73},
	year = {2001}}

@article{petrov1972independent,
	author = {Petrov, VV},
	journal = {Yu. V. Prokhorov. V. StatuleviCius (Eds.)},
	publisher = {Springer},
	title = {Of independent random variables},
	year = {1972}}

@article {MR4550215,
	AUTHOR = {Erd\H{o}s, L\'aszl\'o{} and Xu, Yuanyuan},
	TITLE = {Small deviation estimates for the largest eigenvalue of
	{W}igner matrices},
	JOURNAL = {Bernoulli},
	FJOURNAL = {Bernoulli. Official Journal of the Bernoulli Society for
	Mathematical Statistics and Probability},
	VOLUME = {29},
	YEAR = {2023},
	NUMBER = {2},
	PAGES = {1063--1079},
	ISSN = {1350-7265,1573-9759},
	MRCLASS = {60B20 (60F10)},
	MRNUMBER = {4550215},
	DOI = {10.3150/22-bej1490},
	URL = {https://doi.org/10.3150/22-bej1490},
	}
	\bibliographystyle{abbrv}

	\end{document}